\documentclass[a4paper, reqno]{amsart}
\usepackage[margin=3.81cm]{geometry}
\usepackage[utf8]{inputenc}
\usepackage[english]{babel}

\usepackage{csquotes}
\usepackage{graphicx}
\usepackage[utf8]{inputenc}
\usepackage[T1]{fontenc}
\usepackage[dvipsnames]{xcolor}
\usepackage{subfiles}

\usepackage[textsize = footnotesize]{todonotes}

\usepackage[mathscr]{eucal}
\usepackage[shortlabels]{enumitem}
\usepackage[super]{nth}
\usepackage{leftindex}

\usepackage[backend=biber,
    style=numeric,
    sorting=nyt,
    giveninits=true,
    isbn=false,
    maxalphanames=4,
    date=year,
    maxbibnames=99,
    backref=false]{biblatex}
\DeclareRobustCommand{\SkipTocEntry}[5]{}

\usepackage{amsmath}
\usepackage{amssymb}

\usepackage{amsthm}
\usepackage{thmtools}
\usepackage{graphicx}

\usepackage{dsfont} 
\usepackage{mathtools}
\usepackage{xparse}
\usepackage{tikz-cd}
\tikzset{
    symbol/.style={%
        draw=none,
        every to/.append style={%
            edge node={node [sloped, allow upside down, auto=false]{$#1$}}}
    }
}
\usepackage{spectralsequences}

\definecolor{leafgreen}{RGB}{48, 138, 3}

\usepackage[pdfborder={0 0 0}]{hyperref}
\hypersetup{
   colorlinks=true,
   allcolors=.,
   bookmarksdepth = 3
}

\usepackage[capitalize, nameinlink]{cleveref}
\crefformat{equation}{#2(#1)#3}
\Crefformat{equation}{#2(#1)#3}
\crefname{section}{Section}{Sections}
\crefname{subsection}{Section}{Sections}
\crefname{subsubsection}{Section}{Sections}
\crefname{subappendix}{Section}{Sections}
\crefname{subsubappendix}{Section}{Sections}

\usepackage{microtype}

\theoremstyle{plain}
\newtheorem{theorem}{Theorem}[section]
\newtheorem{introtheorem}{Theorem}
\newtheorem{introcorollary}[introtheorem]{Corollary}
\newtheorem{lemma}[theorem]{Lemma}
\newtheorem{proposition}[theorem]{Proposition}
\newtheorem{corollary}[theorem]{Corollary}

\newtheorem*{theorem*}{Theorem}

\theoremstyle{definition}
\newtheorem{definition}[theorem]{Definition}
\newtheorem{remark}[theorem]{Remark}
\newtheorem{example}[theorem]{Example}

\newtheorem{construction}[theorem]{Construction}

\theoremstyle{remark}

\DeclareMathOperator{\End}{End}

\DeclareMathOperator{\id}{id}
\DeclareMathOperator*{\colim}{colim}

\DeclareMathOperator{\ev}{ev}

\DeclareMathOperator{\free}{free}

\DeclareMathOperator{\forget}{forget}
\DeclareMathOperator{\cofree}{cofree}
\DeclareMathOperator{\triv}{triv}
\DeclareMathOperator{\indec}{indec}
\DeclareMathOperator{\prim}{prim}

\DeclareMathOperator{\LMod}{LMod}
\DeclareMathOperator{\RMod}{RMod}
\DeclareMathOperator{\BMod}{BiMod}

\DeclareMathOperator{\LcoMod}{LcoMod}

\DeclareMathOperator{\Alg}{Alg}
\DeclareMathOperator{\coAlg}{coAlg}

\DeclareMathOperator{\coCAlg}{coCAlg}
\DeclareMathOperator{\Map}{Map}
\DeclareMathOperator{\Sym}{Sym}

\DeclareMathOperator{\cof}{cof}

\DeclareMathOperator{\cross}{cr}

\DeclareMathOperator{\Bahr}{Bar}
\DeclareMathOperator{\Cobar}{Cobar}
\DeclareMathOperator{\Tot}{Tot}

\DeclareMathOperator{\coOp}{coOp}

\newcommand{\hide}[1]{}

\newlist{numberenum}{enumerate}{1}
\setlist[numberenum]{\upshape(\arabic*)}

\newcommand{\Fun}{\mathrm{Fun}}
\newcommand{\CMonnu}{\mathrm{CMon}^{\mathrm{nu}}}

\newcommand{\FunL}{\mathrm{Fun}^{\mathrm{L}}}

\newcommand{\FunLM}{\Fun_{\lten}}
\newcommand{\FunRM}{\Fun_{\rten}}
\newcommand{\FunBM}{\Fun_{\bten}}
\newcommand{\iFunLM}{\underline\Fun_{\lten}}
\newcommand{\iFunRM}{\underline\Fun_{\rten}}
\newcommand{\FunLL}[1]{\FunLM}
\newcommand{\FunRR}[1]{\FunRM}
\newcommand{\FunBB}[2]{\FunBM}
\newcommand{\iFunL}[1]{\iFunLM}
\newcommand{\iFunR}[1]{\iFunRM}

\newcommand{\Fin}{\mathrm{Fin}}
\newcommand{\Surj}{\mathrm{Surj}}

\newcommand{\sA}{\mathscr{A}}
\newcommand{\sB}{\mathscr{B}}
\newcommand{\sC}{\mathscr{C}}
\newcommand{\sD}{\mathscr{D}}
\newcommand{\sE}{\mathscr{E}}

\newcommand{\sP}{\mathscr{P}}
\newcommand{\sQ}{\mathscr{Q}}
\newcommand{\sM}{\mathscr{M}}

\newcommand{\sO}{\mathscr{O}}

\newcommand{\sX}{\mathscr{X}}

\newcommand{\Sp}{\mathrm{Sp}}

\newcommand{\diff}{\mathscr{D}\mathrm{iff}}

\newcommand{\presl}{\mathscr{P}\mathrm{r}^{\mathrm{L}}}

\newcommand{\Spc}{\mathscr{S}}

\newcommand{\Cat}{\mathscr{C}\mathrm{at}}

\newcommand{\MMor}{\mathscr{M}\mathrm{or}}

\newcommand{\Lie}{\mathrm{Lie}}

\newcommand{\op}{\mathrm{op}}

\newcommand{\SymFunL}{\mathrm{SymFun}^\mathrm{L}}

\newcommand{\SSeq}{\mathrm{SSeq}}
\newcommand{\SSeqnu}{\mathrm{SSeq}^{\mathrm{nu}}}

\newcommand{\mSSeqnu}[1]{\mathrm{SSeq}^{\mathrm{nu}}(\Sp^{\times #1}, \Sp)}
\newcommand{\biSSeq}{\mathrm{BSSeq}^{\mathrm{nu}}}

\newcommand{\Tw}{\mathrm{Tw}}

\newcommand{\Ar}{\mathrm{Ar}}

\newcommand{\mlin}{P_{\vec{1}}}

\newcommand{\otimesaug}{\otimes^{\mathrm{aug}}}
\newcommand{\coAlgdpnil}{\mathrm{coAlg}^{\mathrm{dp, nil}}}
\newcommand{\otimesday}{\otimes^{\mathrm{Day}}}
\newcommand{\reAlg}{\Lambda_{\mathrm{Alg}}}
\newcommand{\reLMod}{\Lambda_{\mathrm{LMod}}}
\newcommand{\maAlg}{\mathcal{G}_{\mathrm{Alg}}}
\newcommand{\maLMod}{\mathcal{G}_{\mathrm{LMod}}}

\newcommand{\LL}{\mathbb{L}}

\newcommand{\E}{\mathbb{E}}
\newcommand{\Sph}{\mathbb{S}}
\newcommand{\unit}{\mathbf{1}}

\makeatletter
\newcommand{\oset}[3][0ex]{%
	\mathrel{\mathop{#3}\limits^{
			\vbox to#1{\kern-2\ex@
				\hbox{$\scriptstyle#2$}\vss}}}}
\makeatother

\tikzset{
	rot90/.style={anchor=south, rotate=90, inner sep=.5mm}
} 

\NewDocumentCommand\derprojlim{e{_}}{\mathchoice
{\varprojlim  \IfValueT{#1}{_{\mathclap{#1}}}{}^{\!1\!}\mathop{}}
{\varprojlim^1  \IfValueT{#1}{_{#1}}}
{\varprojlim^1  \IfValueT{#1}{_{#1}}}
{\varprojlim^1  \IfValueT{#1}{_{#1}}}}

\newcommand{\Op}{\mathrm{Op}}

\newcommand{\Com}{\mathbf{Com}}
\newcommand{\coCom}{\mathbf{Com}^{\vee}}

\newcommand{\lten}{\mathrm{LM}}
\newcommand{\rten}{\mathrm{RM}}
\newcommand{\bten}{\mathrm{BM}}

\makeatletter
 \def\subsection{\@startsection{subsection}{1}%
 \z@{.7\linespacing\@plus\linespacing}{.5\linespacing}%
 {\normalfont\bfseries\centering}}
\makeatother

\renewcommand{\epsilon}{\varepsilon}

\begin{document}

\title{A characterization of the spectral Lie operad}

\author{Max Blans}
\address{University of Oxford}
\email{max.blans@maths.ox.ac.uk}

\author{Gijs Heuts}
\address{Utrecht University}
\email{g.s.k.s.heuts@uu.nl}

\begin{abstract}
In this paper we study the structure of the $\infty$-category of spectral Lie algebras. We show that this $\infty$-category admits an interesting symmetric monoidal structure, defined by an analog of the smash product of pointed spaces, and that the free Lie algebra functor $\Sp \to \Lie(\Sp)$ is symmetric monoidal with respect to it. Moreover, this property of the free functor essentially characterizes the spectral Lie operad (among nonunital operads in spectra). This result may be thought of as Koszul dual to the more familiar fact that the free commutative algebra functor takes direct sums to tensor products. One of the key ideas is that the $\infty$-category of spectral Lie algebras behaves in many ways like the $\infty$-category of pointed spaces. More precisely, we deduce structural facts about spectral Lie algebras from familiar statements about spaces by differentiating, in the sense of Goodwillie calculus. The tool to do this is the highly structured generalization of Arone--Ching's chain rule established by Blans--Blom. Numerous other features of spectral Lie algebras follow as well, such as a version of Mather's second cube lemma, the relation between the James construction and loop-suspensions, the Hilton-Milnor splitting, and a version of the EHP sequence.
\end{abstract}

\maketitle

\tableofcontents

\section{Introduction}

In this paper we establish many structural features of \emph{spectral Lie algebras}, which form an extension of the theory of Lie algebras to the setting of higher algebra. This extension arises naturally for at least two reasons: it is `Koszul dual' to the theory of commutative ring spectra \cite{ChingThesis}, which makes it an important notion in the context of deformation theory \cite{brantnermathew,lurieDAGX}, and the theory of spectral Lie algebras bears a close resemblance to the homotopy theory of pointed spaces. This second fact is illustrated by Quillen's Lie algebra model for rational homotopy theory, which can be extended to also give Lie algebra models for the $v_n$-periodic localizations of unstable homotopy theory \cite{heutsLie}.

To begin describing our results, we recall the notion of smash product. For two objects $X$ and $Y$ of a pointed $\infty$-category with finite products and finite colimits, their \emph{smash product} is defined to be the cofiber of the natural map $X \amalg Y \to X \times Y$. This does not generally define a symmetric monoidal structure; indeed, at this level of generality the smash product may fail to be associative. To describe our first result, we write $\LL$ for the spectral Lie operad \cite{ChingThesis}, $\Lie(\Sp)$ for the $\infty$-category of $\LL$-algebras in $\Sp$, and we recall that a spectral Lie algebra $X \in \Lie(\Sp)$ in particular carries a natural \emph{Lie bracket} 
\[
[\cdot, \cdot]\colon \Sigma^{-1} X^{\otimes 2} = \LL(2) \otimes X^{\otimes 2} \to X.
\]

\begin{introtheorem}
\label{thm: mainA}
    For $X, Y \in \Lie(\Sp)$ there is a natural cofiber sequence
    \[
        \Sigma X \wedge Y \to \Sigma X \vee \Sigma Y \to \Sigma X \times \Sigma Y,
    \]
    where the first map is the Whitehead bracket (see \cref{def: whitehead-bracket}) of the inclusions of $\Sigma X$ and $\Sigma Y$ in the wedge and the second map is the canonical map from the coproduct to the product. In case $X$ and $Y$ are free Lie algebras $\free_{\LL}(A)$ and $\free_{\LL}(B)$ respectively, this cofiber sequence is equivalent to
    \[
    \free_{\LL}(\Sigma^{-1} A \otimes B) \xrightarrow{[\cdot,\cdot]} \free_{\LL}(A \oplus B) \to \free_{\LL}(A) \times \free_{\LL}(B),
    \]
    where the map denoted $[\cdot, \cdot]$ is the map of spectral Lie algebras induced by
    \[
    [\iota_{A}, \iota_{B}] \colon \Sigma^{-1} A \otimes B \to \free_{\LL}(A \oplus B),
    \]
    the Lie bracket of the inclusions of $A$ and $B$ into $\free_{\LL}(A \oplus B)$.
\end{introtheorem}

\begin{remark}
An interpretation of the second cofiber sequence in perhaps more familiar terms is the following. One can intuitively think of $\free_{\LL}(\Sigma^{-1} A \otimes B)$ as the `commutator ideal' of $\free_{\LL}(A \oplus B)$ generated by brackets of classes in $A$ with classes in $B$. Quotienting by this ideal then yields the product of $\free_{\LL}(A)$ and $\free_{\LL}(B)$, in which $A$ and $B$ commute. For this reason, we will also refer to both cofiber sequences of \Cref{thm: mainA} as the \emph{commutator cofiber sequence}.
\end{remark}

One consequence of \Cref{thm: mainA} is that the smash product on $\Lie(\Sp)$ is in fact associative:

\begin{introcorollary} \label{cor: intro-smash-associative}
The smash product defines a symmetric monoidal structure on $\Lie(\Sp)$ which preserves colimits separately in each variable. Moreover, the cartesian product
\[
    \times \colon \Lie(\Sp) \times \Lie(\Sp) \to \Lie(\Sp) 
\]
preserves colimits indexed by weakly contractible categories separately in each variable.
\end{introcorollary}

\Cref{thm: mainA} strongly suggests that the free Lie algebra functor sends tensor products of spectra to smash products of Lie algebras. This is true; better yet, \Cref{introthm: characterization} shows that this property of the free algebra functor characterizes the spectral Lie operad among nonunital operads. To state this result precisely, we first fix some notation. For a nonunital operad $\sO$ in spectra we write $B\sO$ for its bar construction, which is naturally a cooperad \cite[Section 3]{heuts2024koszulduality}. We write $\mathrm{cot}_{\sO}$ for the left adjoint of the trivial $\sO$-algebra functor; the notation references the terminology \emph{cotangent fiber} (see \cite[Section 6]{heuts2024koszulduality}), but other common names are \emph{topological Quillen homology} and \emph{derived indecomposables}. 

\begin{introtheorem} \label{introthm: characterization}
Let $\sO$ be a nonunital and reduced operad in spectra. Then the following are equivalent:
\begin{enumerate}[\upshape{(}\arabic*\upshape{)}]
    \item The operad $\sO$ is equivalent to the spectral Lie operad $\LL$.
    \item The cooperad $B\sO$ is equivalent to the nonunital cocommutative cooperad $\mathbf{Com}^{\vee}$.
    \item
    The free $\sO$-algebra functor 
    \[
    \free_{\sO} \colon \Sp \to \Alg_{\sO}(\Sp)
    \] 
    admits a nonunital symmetric monoidal structure with respect to the tensor product on $\Sp$ and the smash product on $\Alg_{\sO}(\Sp)$.
    \item The functor
    \[
        \cot_{\sO} \colon \Alg_{\sO}(\Sp) \to \Sp
    \]
    admits a nonunital symmetric monoidal structure with respect to the smash product on $\Alg_{\sO}(\Sp)$ and the tensor product on $\Sp$.
    \item The functor
    \[
    \forget \colon \coAlg^{\mathrm{dp, nil}}_{B\sO}(\Sp) \to \Sp
    \]
    admits a nonunital symmetric monoidal structure with respect to the smash product on $\coAlg^{\mathrm{dp, nil}}_{B\sO}(\Sp)$ and the tensor product on $\Sp$.
\end{enumerate}
\end{introtheorem}

\begin{remark}
When formulating item (3), we cannot assume that the smash product on $\Alg_{\sO}(\Sp)$ is associative. For a general operad $\sO$ it only defines an \emph{oplax} nonunital symmetric monoidal structure on this $\infty$-category. However, this is enough to define what it means for a functor from a symmetric monoidal $\infty$-category to $\Alg_{\sO}(\Sp)$ to admit a nonunital symmetric monoidal structure, so that (3) is well-posed. Moreover, we will see that if the conclusion of (3) is satisfied, then in fact the smash product on $\Alg_{\sO}(\Sp)$ is associative and does define an honest nonunital symmetric monoidal structure.
See \cref{sec: smash} for details on this technical point.
\end{remark}

\begin{remark}
Item (4) can equivalently be formulated in terms of the cartesian product as follows:
\begin{itemize}
\item[(4*)] The functor
\[
(\mathrm{cot}_{\sO})_+\colon \Alg_{\sO}(\Sp) \to \Sp \colon X \mapsto \Sph \oplus \mathrm{cot}_{\sO}(X)
\]
admits a symmetric monoidal structure with respect to the cartesian product on $\Alg_{\sO}(\Sp)$ and the tensor product on $\Sp$.
\end{itemize}
A useful analogy is to compare $\mathrm{cot}_{\sO}$ and $(\mathrm{cot}_{\sO})_+$ with the functors $\Sigma^\infty$ and $\Sigma^\infty_+$ on the $\infty$-category of pointed spaces.
\end{remark}

Our main technique for proving \Cref{thm: mainA} and \Cref{introthm: characterization} is the chain rule in Goodwillie calculus \cite{AroneChingChainRule,blansblom2025chainrulegoodwilliecalculus}; we review what we need about it in \cref{sec: Calculus}. In particular, this chain rule yields an operad structure on the derivatives of the identity functor $\partial_*{\id_{\Spc_*}}$ on the $\infty$-category $\Spc_*$ of pointed spaces. This operad is Koszul dual to the nonunital commutative operad  and therefore (by definition) equivalent to the spectral Lie operad $\LL$ (see \cref{prop: derivatives-id-Lie}). Moreover, the derivatives $\partial_* F$ of any functor $F\colon \Spc_* \to \Spc_*$ will form a bimodule over the operad $\LL$. This allows us to deduce results about $\LL$-bimodules by differentiating familiar functors on pointed spaces, for example the ones in the analog of the cofiber sequence of \Cref{thm: mainA} in the $\infty$-category $\Spc_*$.

The use of Goodwillie calculus to deduce facts about spectral Lie algebras from statements about pointed spaces is a very productive technique. To illustrate it further, we prove the three theorems below. We stress that this collection is not exhaustive; many classical theorems from unstable homotopy theory admit Lie algebra analogs using similar arguments.

Write $\mathbf{\Delta}^{\mathrm{inj}}$ for the subcategory of the usual simplex category $\mathbf{\Delta}$ containing only injective morphisms. Then a spectral Lie algebra $X$ defines a diagram
\[
\mathcal{J}_\bullet(X) \colon \mathbf{\Delta}^{\mathrm{inj}} \to \Lie(\Sp)\colon {[n]} \mapsto X^{n+1},
\]
which sends an injective map $\alpha\colon {[m]} \to {[n]}$ to the evident inclusion $X^{m+1} \to X^{n+1}$ inserting 0 in any coordinate outside the image of $\alpha$ (where we interpret $X^{n+1}$ as the product indexed by the set ${[n]} = \{0, \ldots, n\}$, and similarly for $m$). We refer to the colimit $\mathcal{J}(X) := \mathrm{colim} \,\mathcal{J}_\bullet(X)$ as the \emph{James construction} of $X$. We will see that $\mathcal{J}(X)$ naturally admits the structure of an $\mathbb{E}_1$-monoid and is in fact the free $\mathbb{E}_1$-monoid generated by $X$. (All of this is completely parallel to the usual James construction of pointed spaces.)

\begin{introtheorem} \label{introthm:James}
For any $X \in \Lie(\Sp)$ there is a natural equivalence of spectral Lie algebras $\mathcal{J}(X) \xrightarrow{\simeq} \Omega\Sigma X$. Moreover, the James construction splits upon suspending once:
\[
\Sigma \mathcal{J}(X) \simeq \coprod_{n \geq 1} \Sigma X^{\wedge n}.
\]
Finally, any $\mathbb{E}_1$-monoid in $\Lie(\Sp)$ is grouplike.
\end{introtheorem}
\begin{remark}
Contrary to the case of pointed spaces, this theorem does not need any assumption on the connectivity of $X$.
\end{remark}

For a spectral Lie algebra $X$, we will write $E\colon X \to \Omega\Sigma X$ for the unit of the loop-suspension adjunction. In \Cref{subsec:EHP} we construct a natural \emph{Hopf map} $H\colon \Omega\Sigma X \to \Omega\Sigma X^{\wedge 2}$. Then we establish the following version of James' EHP sequence:

\begin{introtheorem} \label{introthm:EHP}
There is a natural nullhomotopy of the composite
\[
X \xrightarrow{E} \Omega\Sigma X \xrightarrow{H} \Omega\Sigma X^{\wedge 2}.
\]
Moreover, when evaluating on a free Lie algebra $X = \free_{\LL}(\Sph^k)$ generated by a shifted sphere, this sequence becomes a fiber sequence upon localization at 2.
\end{introtheorem}

We prove these results, as well as an analog of the Hilton--Milnor splitting, in \cref{sec: homotopy-theory-of-lie-algebras}. Our final result concerns Mather's second cube lemma: this classical result states that homotopy pushouts of spaces are `universal', in the sense that they are preserved by pullback along any map. This statement also holds in the $\infty$-category of spectral Lie algebras:

\begin{introtheorem} \label{introthm:Mather}
Weakly contractible colimits are universal in the $\infty$-category of spectral Lie algebras. More precisely, let $X \in \Lie(\Sp)$, let $\mathcal{I}$ be a weakly contractible small $\infty$-category, and let $g\colon \mathcal{I} \to \Lie(\Sp)_{/X}$ be a diagram. Then the colimit of $g$ is stable under base change, i.e., for any $f\colon Y \to X$ the natural comparison map
\[
\mathrm{colim}_{\mathcal{I}} (Y \times_X g(-)) \to Y \times_X \mathrm{colim}_{\mathcal{I}} g(-)
\]
is an equivalence.
\end{introtheorem}

\subsection*{Acknowledgments}

The first author thanks Thomas Blom for many helpful discussions.
The second author thanks Jacob Lurie for conversations related to \Cref{thm: mainA}. This work was supported by an ERC Starting Grant (no. 950048) and an NWO Vidi grant (no. 233.093).
The first author was also supported by the Royal Society through grant URF\textbackslash R1\textbackslash211075.

\section{Operads}

In this section we recall some definitions and results relating to operads, cooperads, and Koszul duality.

\subsection{Symmetric sequences and operads}

\begin{definition}
The $\infty$-category of \emph{symmetric sequences} in spectra is defined to be 
\[
\SSeq(\Sp) \coloneqq \Fun(\Fin^{\cong}, \Sp),
\]
where $\Fin^{\cong}$ denotes the category of finite sets and bijections.  
The $\infty$-category of \emph{nonunital symmetric sequences} $\SSeqnu(\Sp)$ is the full subcategory of $\SSeq(\Sp)$ spanned by the symmetric sequences that send the empty set to $0$.
\end{definition}

For $A \in \SSeq(\Sp)$ and a finite set $I$, we write $A_I$ for the value of $A$ on $I$.
We write $A_n$ for the spectrum with $\Sigma_n$-action given by evaluating $A$ on the set $\{1, \ldots, n\}$.
We call $A_n$ the arity $n$ term of $A$.
There is a fully faithful embedding $\Sp \hookrightarrow \SSeq(\Sp)$ with essential image those symmetric sequences $A$ concentrated in arity $0$.

The $\infty$-category $\SSeq(\Sp)$ comes equipped with a monoidal structure, called the \emph{composition product}; see \cite{brantnerThesis} for a construction.
Given $A, B \in \SSeqnu(\Sp)$, their composition product $A \circ B$ is the nonunital symmetric sequence whose value on a finite set $I$ is
\[
(A \circ B)_I = \bigoplus_{E \in \mathrm{Part}(I)} A_E \otimes \bigotimes_{J \in E} B_J,
\]
where $\mathrm{Part}(I)$ denotes the set of partitions of $I$.
We denote the unit for the composition product by $\unit$, which is given by
\[
\unit_n \simeq \begin{cases}
    \Sph & \text{if $n = 1$,}\\
    0 & \text{otherwise.}
    \end{cases}
\]
If $A \in \SSeqnu(\Sp)$ and $X$ is a spectrum regarded as a symmetric sequence concentrated in arity $0$, then $A \circ X$ is again concentrated in arity $0$, and we have the formula
\[
A \circ X = \bigoplus_{n \geq 1} (A_n \otimes X^{\otimes n})_{h\Sigma_n}.
\]
In this case, we also write $\Sym_A(X) \coloneqq A \circ X$.
This defines a functor $\Sym \colon \SSeq(\Sp) \to \End(\Sp)$, 
which by construction is monoidal with respect to the composition product on the source and functor composition on the target.

\begin{definition}
An \emph{operad} (resp.\ \emph{cooperad}) in spectra is an algebra (resp.\ coalgebra) in $\SSeq(\Sp)$ for the composition product.
We say an operad or cooperad is \emph{nonunital} if its underlying symmetric sequence lies in $\SSeqnu(\Sp)$.
We say a nonunital operad $\sO$ is \emph{reduced} if $\sO_1 = \Sph$, and similarly for cooperads.
\end{definition}

As we only consider reduced operads and cooperads throughout this paper, we will simply refer to them as operads and cooperads.
We write $\Op(\Sp)$ (resp.\ $\coOp(\Sp)$) for the $\infty$-category of operads (resp.\ cooperads).

\begin{definition}
	Let $\sO \in \Op(\Sp)$ and $\sQ \in \mathrm{coOp}(\Sp)$. An \emph{$\sO$-algebra} in $\Sp$ is an algebra for the monad $\Sym_\sO$. 
	A \emph{divided power conilpotent $\mathscr{Q}$-coalgebra} in $\Sp$ is a coalgebra for the comonad $\Sym_{\sQ}$.
		We write $\Alg_\sO(\Sp)$ for the category of $\sO$-algebras and $\coAlgdpnil_{\sQ}(\Sp)$ for the category of conilpotent divided power $\sQ$-coalgebras.
\end{definition}

\begin{remark}
    The reason for the terminology `divided power conilpotent coalgebra' is as follows.
	Suppose $\sQ \in \coOp(\Sp)$ and $X \in \coAlgdpnil_{\sQ}(\Sp)$.
    As $X$ is by definition a coalgebra for the comonad $\Sym_\sQ$, there is a structure morphism
    \[
        X \to \bigoplus_{n \geq 1} (\sQ_n \otimes X^{\otimes n})_{h\Sigma_n}.
    \]
    The presence of the $\Sigma_n$-orbits instead of fixed points accounts for the word ``divided power''.
    The presence of a direct sum instead of a product accounts for the word ``conilpotent''.
\end{remark}

\begin{definition}
    Let $\sO \in \Op(\Sp)$ and $\sQ \in \mathrm{coOp}(\Sp)$.
    Write $\LMod_\sO$ (resp.\ $\LcoMod_\sQ$) for the $\infty$-category of left $\sO$-modules (resp.\ left $\sQ$-comodules) in $\SSeqnu(\Sp)$.
\end{definition}

\begin{remark}
    The $\infty$-category of $\sO$-algebras is equivalent to the $\infty$-category of left $\sO$-modules in $\SSeq(\Sp)$ concentrated in arity $0$. We point this out to emphasize the similarity between the $\infty$-categories $\Alg_\sO(\Sp)$ and $\LMod_\sO$.
\end{remark}

Apart from the composition product, the $\infty$-category $\SSeq(\Sp)$ also comes equipped with a symmetric monoidal structure given by Day convolution.
We think of this as an extension of the tensor product of spectra and denote the Day convolution of symmetric sequences $A$ and $B$ by $A \otimes B$.
It is given by the formula
\[
(A \otimes B)_I \simeq \bigoplus_{I = I_1 \sqcup I_2} A_{I_1} \otimes A_{I_2},
\]
where $I$ is a finite set and the direct sum on the right is indexed by partitions of $I$ into two disjoint subsets. The Day convolution product commutes with colimits in both variables separately. The unit for Day convolution is the sphere spectrum concentrated in arity $0$. Note that $\SSeqnu(\Sp)$ is closed under $\otimes$, so that Day convolution gives this $\infty$-category a nonunital symmetric monoidal structure.

\begin{remark} \label{rem: day-convolution-and-composition}
    Day convolution is in fact closely related to the composition product. As $(\SSeq(\Sp), \otimes)$ is the free stable presentably symmetric monoidal $\infty$-category on one object, there is an equivalence of $\infty$-categories
    \[
    \Fun^{\otimes, \mathrm{L}}(\SSeq(\Sp), \SSeq(\Sp)) \simeq \SSeq(\Sp),
    \]
    where the left hand side denotes the $\infty$-category of colimit preserving symmetric monoidal functors.
    Transporting the reverse of the composition monoidal structure on the left hand side to the right hand side yields the composition product on $\SSeq(\Sp)$, see \cite{BrantnerCamposNuiten,brantnerThesis}.
    In particular, for any $A \in \SSeq(\Sp)$ the functor $B \mapsto B \circ A$ is colimit preserving and strong monoidal for Day convolution.
\end{remark}

\subsection{The cocommutative cooperad}

In this section, we give a construction of the nonunital cocommutative cooperad.
We start with a number of preliminary results on nonunital cocommutative coalgebras.

In dealing with nonunital algebras and coalgebras, it is convenient to work with a slightly non-standard symmetric monoidal structure on spectra and symmetric sequences.
Given a stable presentably symmetric monoidal $\infty$-category $\sC$ with unit $\unit$, there is a symmetric monoidal structure on the double slice $\infty$-category $\sC_{\unit//\unit}$ such that the forgetful functor to $\sC$ is strong monoidal \cite[Remark 2.2.2.5]{HA}.
The $\infty$-category $\sC_{\unit//\unit}$ is equivalent to $\sC$ by the functor that sends an object $\unit \to X \to \unit$ to the cofiber of the first (or equivalently the fiber of the second) map.
By transport of structure, the tensor product on $\sC_{\unit//\unit}$ yields a new symmetric monoidal structure on $\sC$.
If we denote the original tensor product of $\sC$ by $\otimes$, we denote the new one by $\otimesaug$ and call it the \emph{augmented tensor product}.
It is described by the formula
\[
X \otimesaug Y = X \oplus Y \oplus X \otimes Y.
\]
Note that $\otimesaug$ has unit given by the zero object of $\sC$.
We write $\coCAlg^{\mathrm{nu}}(\sC, \otimes)$ for the category of nonunital cocommutative coalgebras in $(\sC, \otimes)$ as defined in \cite[Definition 5.4.4.1]{HA}.

\begin{proposition}
    There is an equivalence of $\infty$-categories 
    \[
    	\coCAlg^{\mathrm{nu}}(\sC, \otimes) \simeq \coCAlg(\sC, \otimesaug).
    \]
\end{proposition}
\begin{proof}
We have the following string of equivalences:
   \[
   \coCAlg(\sC, \otimesaug) \simeq \coCAlg(\sC_{\unit//\unit}, \otimes) \simeq \coCAlg^{\mathrm{aug}}(\sC, \otimes) \simeq \coCAlg^{\mathrm{nu}}(\sC, \otimes),
   \]
   The third $\infty$-category denotes the category of coaugmented cocommutative coalgebras in $\sC$, and the third equivalence is \cite[Proposition 5.4.4.10]{HA}.
\end{proof}

We will use the augmented tensor product in the cases where $\sC$ is either $\Sp$ with the standard tensor product, or $\SSeq(\Sp)$ with the Day convolution product.
Note that $\SSeqnu(\Sp)$ is closed under the augmented Day convolution tensor product.

The following standard result says that the cartesian product on the category of nonunital cocommutative coalgebras is given by $\otimesaug$.

\begin{proposition} \label{prop: product-nu-cocalg}
    Let $\sC$ be a stable presentably symmetric monoidal category.
    The forgetful functor $\coCAlg^{\mathrm{nu}}(\sC) \to \sC$ is strong monoidal from the cartesian product to $\otimesaug$.
\end{proposition}
\begin{proof}
    This follows from the equivalence $\coCAlg^{\mathrm{nu}}(\sC) \simeq \coCAlg(\sC, \otimesaug)$ together with \cite[Proposition 3.2.4.7]{HA}.
\end{proof}

Recall that if $\sC$ is a stable presentably symmetric monoidal $\infty$-category, then the free nonunital commutative algebra on $X \in \sC$ is given by the formula
\[
\Sym(X) \simeq \bigoplus_{n \geq 1} X^{\otimes n}_{h\Sigma_n}.
\]
The cofree cocommutative coalgebra functor on the other hand does not admit such a straightforward formula in general.
However, in the category $\SSeqnu(\Sp)$ we have the pleasant circumstance that it is given by the same formula.
We let $\coCAlg^{\mathrm{nu}}(\SSeqnu(\Sp))$ denote the category of nonunital cocommutative coalgebras in $\SSeqnu(\Sp)$ with respect to Day convolution.

\begin{proposition} \label{prop: cofree-is-sym}
    The underlying functor of the comonad associated with the adjunction
    \[
    \begin{tikzcd}
        \forget \colon \coCAlg^{\mathrm{nu}}(\SSeqnu(\Sp)) \ar[r, shift left] & \SSeqnu(\Sp) \colon \cofree \ar[l, shift left]
    \end{tikzcd}
    \]
    is given by the functor $\Sym \colon \SSeqnu(\Sp) \to \SSeqnu(\Sp)$.
\end{proposition}
\begin{proof}
It follows from \cite[Proposition 3.1.3.3]{HA} that it suffices to show that for every $X \in \SSeqnu(\Sp)$, the Day convolution tensor product $\otimes$ commutes with the limit of the diagram $\Fin^{\cong} \to \SSeqnu(\Sp)$ that sends $\{1, \ldots, n\}$ to the $\Sigma_n$-object $X^{\otimes n}$.
In other words, it is the diagram with limit given by
\[
\prod_{n \geq 1} (X^{\otimes n})^{h \Sigma_n}.
\]
By inspecting the formula for the Day convolution product, we find that this product is levelwise finite, so that it is equivalent to a direct sum.
We also see that the $\Sigma_n$-object $X^{\otimes n}$ is levelwise free, so that its fixed points are equivalent to its orbits.
These equivalences still hold after tensoring the diagram with an arbitrary symmetric sequence.
As the Day convolution product commutes with all colimits in both variables, this completes the proof.
\end{proof}

We now give a construction of the cocommutative cooperad.

\begin{proposition} \label{prop: construction-cocom-cooperad}
    There exists a unique cooperad $\sQ \in \coOp(\Sp)$ such that there is an equivalence
    \[
    \LcoMod_{\sQ} \simeq \coCAlg^{\mathrm{nu}}(\SSeqnu(\Sp))
    \]
    commuting with the forgetful functors.
    This cooperad satisfies $\sQ_n \simeq \Sph$ for all $n \geq 1$.
\end{proposition}
\begin{proof}
    The construction is dual to the one given in \cite[Construction 4.1.6]{raksit2020hochschildhomologyderivedrham}.
    We reproduce it here for the convenience of the reader.

    As the composition product of symmetric sequences distributes over Day convolution in the first variable (see \cref{rem: day-convolution-and-composition}), it follows that $\coCAlg^{\mathrm{nu}}(\SSeqnu(\Sp))$ admits a right action from $(\SSeqnu(\Sp), \circ)$ such that the forgetful functor $\coCAlg^{\mathrm{nu}}(\SSeqnu(\Sp)) \to \SSeqnu(\Sp)$ is $\SSeqnu(\Sp)$-linear.
    It follows that its right adjoint, the $\cofree$ functor, admits a lax $\SSeqnu(\Sp)$-linear structure by \cite[Corollary 7.3.2.7]{HA}.
    One can use the explicit formula for this functor from \cref{prop: cofree-is-sym} to show that it is in fact strong $\SSeqnu(\Sp)$-linear.
    Indeed, this amounts to the observation that for any $A, B \in \SSeqnu(\Sp)$, the composite
    \[
    \bigoplus_{n \geq 1} A^{\otimes n}_{h\Sigma_n} \circ B \to \bigoplus_{m \geq 1} \left [ \bigoplus_{n \geq 1} A^{\otimes n}_{h\Sigma_n} \circ B \right ]^{\otimes m}_{h\Sigma_m} \to \bigoplus_{m \geq 1} (A \circ B)^{\otimes m}_{h\Sigma_m}
    \]
    is an equivalence, where the first map is the comultiplication of $\cofree(A) \circ B$, and the second map is the projection onto the $n=1$ components.

    The nonunital cocommutative coalgebra comonad therefore admits a right $\SSeqnu(\Sp)$-linear structure, so that it has to be of the form $A \mapsto \sQ \circ A$ for a unique cooperad $\sQ$.
    The fact that $\sQ_n \simeq\Sph$ for all $n \geq 1$ follows from the natural equivalence $\Sym(A) \simeq \sQ \circ A$.
\end{proof}

\begin{definition}
    We define the cocommutative cooperad $\coCom$ to be the cooperad from \Cref{prop: construction-cocom-cooperad}.
\end{definition}

\subsection{Koszul duality}
We here recall some basic results on Koszul duality of operads and cooperads.

\begin{definition}
Suppose $\sO$ is an operad in spectra, $M$ is a left $\sO$-module, and $N$ is a right $\sO$-module in $\SSeq(\Sp)$.
The \emph{bar construction} $B(M, \sO, N)$ is the geometric realization of the usual simplicial diagram
\[
\begin{tikzcd}
          \cdots M \circ \sO \circ \sO \circ N \ar[r, shift left=2] \ar[r] \ar[r, shift right=2] & M \circ \sO \circ N \ar[l, shift left, shorten=0.4em] \ar[l, shift right, shorten=0.4em] \ar[r, shift left] \ar[r, shift right] & M \circ N. \ar[l, shorten=0.4em]
   \end{tikzcd}
\]
For $\sQ \in \coOp(\Sp)$ and left and right $Q$-comodules $M$ and $N$, the \emph{cobar construction} $C(M, Q, N)$ is given by the totalization of the analogous cosimplicial diagram.
We will sometimes write $M \circ_\sO N$ for the bar construction $B(M, \sO, N)$.
\end{definition}

If $\sO$ is an operad, the projection $\sO \to \unit$ onto the arity 1 term provides $\sO$ with a canonical augmentation and makes $\unit$ into an $\sO$-bimodule.
Similarly, the inclusion of the arity $1$ term gives any cooperad a canonical coaugmentation.

\begin{proposition}
\label{prop:KDoperads}
	There is an adjoint equivalence of $\infty$-categories
	\[
	\begin{tikzcd}[sep = large]
		\Op(\Sp) \ar[r, "B", shift left] & \coOp(\Sp) \ar[l, "C", shift left]
	\end{tikzcd}
	\]
	given on underlying symmetric sequences by $B\sO = B(\unit, \sO, \unit)$ and $C\sQ = C(\unit, \sQ, \unit)$. 
\end{proposition}
\begin{proof}
    This is originally due to Ching \cite{ChingBarCobar}.
    A transposition of this proof to the setting of $\infty$-categories was given in \cite[Theorem 3.4]{heuts2024koszulduality}.
\end{proof}

We call $B\sO$ and $C\sQ$ the Koszul duals of $\sO$ and $\sQ$, respectively.

\begin{definition} \label{def: spectral-Lie-operad}
    The \emph{spectral Lie operad} $\LL$ is the Koszul dual of the cocommutative cooperad $\Com^{\vee}$.
\end{definition}

We will now discuss Koszul duality between $\sO$-modules and $B\sO$-comodules.
Let $\sM$ be either $\Sp$ or $\SSeqnu(\Sp)$ equipped with the left action by $\SSeqnu(\Sp)$ coming from the composition product.
For every operad $\sO \in \Op(\Sp)$ we have a pair of morphisms
\[
\unit \to \sO \to \unit
\]
given by the unit and the augmentation.
By restriction and extension of scalars along these maps, we obtain a pair of adjunctions (left adjoints on top)
\[
\begin{tikzcd}[sep = large]
	\sM \ar[r, shift left, "\free_\sO"] & \LMod_\sO(\sM) \ar[r, shift left, "\cot_\sO"] \ar[l, shift left, "\forget_\sO"] & \sM. \ar[l, shift left, "\triv_\sO"]
\end{tikzcd}
\]
(Here $\LMod_\sO(\sM)$ is supposed to be read as $\Alg_{\sO}(\Sp)$ if $\sM = \Sp$.) Observe that the left adjoints as well as the right adjoints compose to the identity functor.
The functor $\cot_\sO$ is called the \emph{cotangent fiber}.
It can be computed by the formula
\[
\cot_\sO(M) \simeq B(\unit, \sO, M).
\]

\begin{proposition}[{\cite[Proposition 6.3]{heuts2024koszulduality}}] \label{prop: cot-triv-stabilization}
    The adjunction $(\cot_\sO, \triv_\sO)$ exhibits $\Sp$ as the stabilization of $\Alg_{\sO}(\Sp)$.
\end{proposition}

\begin{proposition}[{\cite[Proposition 3.28]{BrantnerCamposNuiten}}]
	The comonad $\cot_\sO \triv_\sO$ is equivalent to $B\sO \circ -$, where $\circ$ denotes the left action of $\SSeqnu(\Sp)$ on $\sM$.
	The adjunction between $\cot_\sO$ and $\triv_\sO$ therefore induces an adjunction
	\[
	\begin{tikzcd}
		\indec_\sO \colon \LMod_\sO(\sM) \ar[r, shift left] & \LcoMod_{B\sO}(\sM) \colon \prim_{B\sO} \ar[l, shift left]
	\end{tikzcd}
	\]
	We have the formulas $\indec_\sO(M) \simeq B(\unit, \sO, M)$ and $\prim_{B\sO}(N) \simeq C(\unit, B\sO, N)$.
\end{proposition}

This adjunction turns out to be an equivalence in case $\sM = \SSeqnu(\Sp)$.

\begin{proposition}[{\cite[Theorem 14.17]{heuts2024koszulduality}}] \label{prop: koszul-duality-left-modules}
	The adjunction between $\indec_\sO$ and $\prim_{B\sO}$ yields an equivalence
	\[
	\LMod_\sO(\SSeqnu(\Sp)) \simeq \LcoMod_{B\sO}(\SSeqnu(\Sp)).
	\]
    The same holds if we replace $\SSeqnu(\Sp)$ by $\RMod_{\sP}(\SSeqnu(\Sp))$ for an operad $\sP$.
\end{proposition}

\begin{remark}
	In contrast to the previous proposition, the adjunction
    \[
    \begin{tikzcd}
        \indec_\sO \colon \Alg_{\sO}(\Sp) \ar[r, shift right] & \coAlgdpnil_{B\sO}(\Sp) \ar[l, shift right] \colon \prim_{B\sO}
    \end{tikzcd}
    \]
    is in general \emph{not} an equivalence of categories.
\end{remark}

The basic feature of the adjunction $(\indec_\sO, \prim_{B\sO})$ is that it interchanges the roles of free and trivial (co)modules.
More precisely, we have the following pair of commutative diagrams (see \cite[Sections 6-7]{heuts2024koszulduality} for a proof):
\[
\begin{tikzcd}
	& \LMod_\sO(\sM) \ar[dr, "\cot_\sO"] \ar[dd, "\indec_\sO"] & & & \LMod_\sO(\sM) \ar[dl, "\forget_\sO"'] & \\
	\sM \ar[dr, "\triv_{B\sO}"'] \ar[ur, "\free_\sO"] & & \sM & \sM & & \sM \ar[dl, "\cofree_{B\sO}"] \ar[ul, "\triv_\sO"'] \\
	& \LcoMod_{B\sO}(\sM) \ar[ur, "\forget_{B\sO}"'] & & & \LcoMod_{B\sO}(\sM) \ar[uu, "\prim_{B\sO}"'] \ar[ul] & 
\end{tikzcd}
\]

All functors in the first diagram are left adjoints, and the second diagram is obtained from the first by passing to right adjoints.
The functor $\triv_{B\sO}$ is the trivial $B\sO$-coalgebra functor.
It can be constructed as the pushforward along $\unit \to B\sO$.

\section{Calculus} \label{sec: Calculus}

The purpose of this section is to explain how Goodwillie calculus and the generalized chain rule from \cite{blansblom2025chainrulegoodwilliecalculus} can be used to study operads and their algebras.

\subsection{Goodwillie derivatives}

We say an $\infty$-category $\sC$ is \emph{differentiable} if it is pointed, presentable, and filtered colimits commute with finite limits in $\sC$.
We say a functor $F \colon \sC \to \sD$ between differentiable $\infty$-categories is \emph{reduced} if it preserves the zero object, and \emph{finitary} if it preserves filtered colimits.
We write $\Fun^{\ast, \omega}(\sC, \sD)$ for the category of reduced finitary functors from $\sC$ to $\sD$.

Suppose that $\sC$ and $\sD$ are differentiable $\infty$-categories with specified equivalences $\Sp(\sC) \simeq \Sp \simeq \Sp(\sD)$.
Then Goodwillie calculus attaches to any reduced finitary functor $F \colon \sC \to \sD$ a nonunital symmetric sequence
\[
\partial_*F \in \SSeqnu(\Sp)
\]
called the \emph{derivatives} of $F$.

Our main tool for studying operads and algebras through Goodwillie calculus is the following theorem, which was one of the main results of \cite{blansblom2025chainrulegoodwilliecalculus}.

\begin{theorem}[{\cite[Theorem 3.3.2]{blansblom2025chainrulegoodwilliecalculus}}] \label{thm: lax-derivatives}
    Let $\sC$ be a differentiable $\infty$-category with a specified equivalence $\Sp(\sC) \simeq \Sp$.
    Then the functor
    \[
    \partial_* \colon \Fun^{\ast, \omega}(\sC, \sC) \to \SSeqnu(\Sp)
    \]
    admits a lax monoidal structure, where the source is equipped with functor composition and the target with the composition product.
    If $\sC \simeq \Sp$, then the functor is strong monoidal.
\end{theorem}

This result implies that $\partial_*$ sends reduced finitary monads to operads.
In particular, $\partial_*{\id_\sC}$ acquires the structure of an operad.
Since any reduced finitary functor $F \colon \sC \to \sC$ is trivially a bimodule over $\id_{\sC}$, it also follows that the derivatives $\partial_*F$ can be equipped with the structure of a $\partial_*{\id_\sC}$-bimodule.

The fact that $\partial_*$ is a strong symmetric monoidal functor for endofunctors of spectra is referred to as the \emph{stable chain rule}.
It implies that $\partial_*(\Sigma^\infty_\sC \Omega^\infty_\sC)$ inherits the structure of a cooperad from the comonad $\Sigma^\infty_\sC \Omega^\infty_\sC$, where we write
\[
\begin{tikzcd}
\Sigma^\infty_\sC \colon \sC \ar[r, shift left] & \Sp \colon \Omega^\infty_\sC \ar[l, shift left]
\end{tikzcd}
\]
for the stabilization adjunction of $\sC$.
The operad $\partial_*{\id_\sC}$ and the cooperad $\partial_*(\Sigma^\infty_\sC\Omega^\infty_\sC)$ are related through Koszul duality:

\begin{proposition}[{\cite[Theorem 4.3.1]{blansblom2025chainrulegoodwilliecalculus}}] \label{prop: Koszul-dual-derivative-id}
    Let $\sC$ be a differentiable $\infty$-category with a specified equivalence $\Sp(\sC) \simeq \Sp$.
    Then there is an isomorphism of operads
    \[
    \partial_*{\id_\sC} \simeq \Cobar \partial_*(\Sigma^\infty_\sC \Omega^\infty_\sC).
    \]
\end{proposition}

\begin{example} \label{ex: der-identity-pointed-spaces}
    The derivatives of the identity functor in pointed spaces are isomorphic to the spectral Lie operad:
    \[
    \partial_*{\id_{\Spc_*}} \simeq \LL.
    \]
    This was first proved in a point set model by Ching in his thesis \cite{ChingThesis}.
    It is equivalent to the statement that $\partial_*(\Sigma^\infty \Omega^\infty)$ is isomorphic to the cocommutative cooperad, which was proved by Arone and Ching in \cite{AroneChingChainRule}, again in a point set model.
    A proof that the lax monoidal functor $\partial_*$ from \cref{thm: lax-derivatives} induces this operad structure has not yet appeared; we will give a proof in \cref{cor: derivatives-identity-pointed-spaces} below using some of the techniques from this paper.
    A different proof will appear in \cite{BlansBlom2026}.
\end{example}

\begin{example}
    Since the identity functor on any stable category is linear, we have
    \[
    \partial_*{\id_\Sp} \simeq \unit.
    \]
    In other words, the derivatives of the identity functor in spectra are isomorphic to the trivial operad.
\end{example}

\begin{example} \label{ex: der-identity-O-algebras}
    Let $\sO$ be an operad in spectra.
    Then the category $\Alg_\sO(\Sp)$ is differentiable and there is an equivalence of operads 
    \[
    \partial_*{\id_{\Alg_{\sO}(\Sp)}} \simeq \sO.
    \]
    In a different setting, this was first proved by Ching in \cite{ChingDayConvolution}.
    A proof that the lax monoidal functor $\partial_*$ from \cref{thm: lax-derivatives} induces this operad structure was given in the thesis of the first author \cite[Theorem 7.2.6]{Blans-Thesis}.
\end{example}

\begin{example} \label{ex: der-retraction-sym}
    It will be proved in \cite{BlansBlom2026} that the composite
    \[
    \begin{tikzcd}
        \SSeqnu(\Sp) \ar[r, "\Sym"] & \End^{\ast, \omega}(\Sp) \ar[r, "\partial_*"] & \SSeqnu(\Sp)
    \end{tikzcd}
    \]
    is equivalent as a strong monoidal functor to the identity functor of $\SSeqnu(\Sp)$.
    In particular, for any operad $\sO$ we have an equivalence $\partial_* \Sym_\sO \simeq \sO$ of operads.
    The same holds for any cooperad.
\end{example}

The derivatives construction admits even more functoriality.
To begin with, the lax structure can (in an appropriate sense) be extended to reduced finitary functors $F \colon \sC \to \sD$ where the source is not necessarily equivalent to the target.
As a consequence, $\partial_*F$ in this case admits the structure of a $(\partial_*{\id_\sD}, \partial_*{\id_\sC})$-bimodule in $\SSeqnu(\Sp)$.
The extra functoriality on $\partial_*$ amounts to the compatibility of this bimodule structure with functor composition.
It can be most conveniently expressed in the language of $(\infty, 2)$-categories, as we will now explain.

Write $\diff_{\Sp}$ for the $(\infty,2)$-category whose objects are differentiable $\infty$-categories $\sC$ with a specified equivalence $\Sp(\sC) \simeq \Sp$ and whose morphisms are reduced finitary functors.
We do not require the morphisms to be compatible with the equivalence $\Sp(\sC) \simeq \Sp$.
Making use of a construction from \cite{blom2025straighteningfunctor}, it was shown in \cite[Section 4.4.1]{blansblom2025chainrulegoodwilliecalculus} that there exists an $(\infty,2)$-precategory\footnote{An $(\infty, 2)$-precategory is a Segal object $\sX \colon \Delta^\op \to \Cat$ such that $\sX_0$ is a space. It is an $(\infty, 2)$-category if this Segal object is additionally complete. The distinction between these notions plays no role in this article.} $\MMor_\Sp$ --- the notation stands for \emph{Morita category} --- whose objects are the collection of operads $\sO \in \Op(\Sp)$.
For any pair of operads $\sO, \sP$, the category of morphisms from $\sO$ to $\sP$ in $\MMor_\Sp$ is given by
\[
\MMor_\Sp(\sO, \sP) \simeq \BMod_{(\sP, \sO)}(\SSeqnu(\Sp)),
\]
i.e.\ the category of $(\sP, \sO)$-bimodules in $\SSeqnu(\Sp)$.
Given a $(\sP, \sO)$-bimodule $N$ and a $(\sQ, \sP)$-bimodule $M$, their composition in $\MMor_\Sp$ is given by the relative composition product $M \circ_\sP N$.

\begin{theorem}[{\cite[Theorem 4.4.7]{blansblom2025chainrulegoodwilliecalculus}}]
    There is a strong functor of $(\infty, 2)$-precategories
    \[
    \partial_* \colon \diff_\Sp \to \MMor_\Sp
    \]
    that sends $\sC$ to the operad $\partial_*{\id_\sC}$.
    It sends a reduced finitary functor $F \colon \sC \to \sD$ to the $(\partial_*{\id_\sD}, \partial_*{\id_\sC})$-bimodule $\partial_*F$.
\end{theorem}

\begin{remark}
The fact that the functor from the previous theorem is a strong rather than lax functor of $(\infty, 2)$-precategories is an alternative formulation of what is commonly known as the chain rule in Goodwillie calculus.
Indeed, given a pair of reduced finitary functors $F \colon \sD \to \sE$ and $G \colon \sC \to \sD$, functoriality of $\partial_*$ amounts to the equivalence
\[
\partial_*(FG) \simeq \partial_*F \circ_{\partial_*{\id_\sD}} \partial_*G,
\]
which is the usual way of writing the chain rule.
\end{remark}

For $\sC, \sD \in \diff_\Sp$, the functor $\partial_* \colon \diff_\Sp \to \MMor_\Sp$ induces on mapping $\infty$-categories a functor
\[
\partial_* \colon \Fun^{\ast, \omega}(\sC, \sD) \to \BMod_{(\partial_*{\id_\sD}, \partial_*{\id_\sC})}(\SSeqnu(\Sp))
\]
that lifts the usual derivatives functor taking values in $\SSeqnu(\Sp)$.
This interacts with Koszul duality in the following way:

\begin{proposition}\label{prop: derivatives-stabilization-indec}
    Let $\sC, \sD \in \diff_\Sp$.
    Then the following diagram commutes
    \[
    \begin{tikzcd}[sep = large]
        \Fun^{\ast, \omega}(\sC, \sD) \ar[r, "\partial_*"] \ar[d, "\Sigma_\sD^\infty \circ -"] & \BMod_{\partial_*{\id_\sD}}(\SSeqnu(\Sp)) \ar[d, "\indec_{\partial_*{\id_\sD}}"] \\
        \LcoMod_{\Sigma^{\infty}_\sD \Omega^\infty_\sD}(\Fun^{\ast, \omega}(\sC, \Sp)) \ar[r, "\partial_*"] & \LcoMod_{\partial_*(\Sigma^\infty_\sD\Omega^\infty_\sD)}(\RMod_{\partial_*{\id_\sC}}(\SSeqnu(\Sp)))
    \end{tikzcd}
    \]
    where we identified the target of the right vertical arrow using the equivalence $\Bahr \partial_*{\id_\sD} \simeq \partial_*(\Sigma^\infty_\sD \Omega^\infty_\sD)$ from \cref{prop: Koszul-dual-derivative-id}.
\end{proposition}
\begin{proof}
    This was proved in \cite[Corollary 4.4.4]{blansblom2025chainrulegoodwilliecalculus} for left modules, but the same proof works for bimodules as well.
\end{proof}

The derivatives functor taking values in $\SSeqnu(\Sp)$ preserves finite limits, but it preserves hardly any colimits.
The following result, which will play an important role throughout the rest of this paper, shows that the situation improves considerably when we let it take values in bimodules.

\begin{proposition} \label{prop: derivatives-lmod-colim}
    Let $\sC, \sD \in \diff_\Sp$.
    Then the functor
    \[
        \partial_* \colon \Fun^{\ast, \omega}(\sC, \sD) \to {\BMod}_{\partial_*{\id_\sD}}(\SSeqnu(\Sp))
    \]
    preserves colimits and finite limits.
\end{proposition}

\begin{remark}
    This proposition and its proof already appeared in the thesis of the first author \cite[Proposition 7.1.1]{Blans-Thesis}.
    The cases where $\sC$ and $\sD$ are either pointed spaces or spectra were already proved by Arone and Ching in \cite{AroneChingClassification} by explicitly constructing a right adjoint to $\partial_*$. 
\end{remark}

\begin{proof}
    We first prove the claim about colimits.
    Since the composition product preserves colimits in the left variable, the forgetful functor
    \[
        {\BMod}_{({\partial_*{\id_\sD}},{\partial_*{\id_\sC}})}(\SSeqnu(\Sp)) \to {\LMod}_{\partial_*{\id_\sD}}(\SSeqnu(\Sp))
    \]
    creates colimits by \cite[Corollary 4.2.3.5]{HA}.
    It therefore suffices to prove the proposition for ${\LMod}_{\partial_*{\id_\sD}}$ instead of ${\BMod}_{({\partial_*{\id_\sD}},{\partial_*{\id_\sC}})}$.
    Consider the commutative square from \cref{prop: derivatives-stabilization-indec}. 
    By \cref{prop: koszul-duality-left-modules}, the right vertical functor is an equivalence of categories.
    As $\Sigma^\infty_{\sD} \circ -$ preserves colimits, it therefore suffices to show that the bottom horizontal arrow 
    \[
    \partial_* \colon \LcoMod_{\Sigma^{\infty}_\sD \Omega^\infty_\sD}(\Fun^{\ast, \omega}(\sC, \Sp)) \to \LcoMod_{\partial_*(\Sigma^\infty_\sD\Omega^\infty_\sD)}(\SSeqnu(\Sp))
    \]
    preserves colimits.
    Since colimits in categories of left comodules are computed on underlying objects by \cite[Corollary 4.2.3.3]{HA}, this in turn follows if we can show
    \[
    \begin{tikzcd}
        \partial_* \colon \Fun^{\ast, \omega}(\sC, \Sp) \ar[r] & \SSeqnu(\Sp)
    \end{tikzcd}
    \]
    preserves colimits.
    This functor factors as
    \[
    \begin{tikzcd}
        \Fun^{\ast, \omega}(\sC, \Sp) \ar[r, "\cross_*"] & \prod_{n \geq 1} \Fun^{*, \omega}(\sC^{\times n}_{h \Sigma_n}, \Sp) \ar[r, "\mlin"] & \SSeqnu(\Sp),
    \end{tikzcd}
    \]
    where $\Fun^{*, \omega}(\sC^{\times n}_{h \Sigma_n}, \Sp)$ denotes the category of functors $\sC^{\times n}_{h \Sigma_n} \to \Sp$ such that the underlying functor $\sC^{\times n} \to \Sp$ is finitary and reduced in each variable separately, $\cross_*$ is the symmetric cross-effects functor, and $\mlin$ is multilinearization.
    It follows immediately from the formula for $\cross_*$ (see \cite[Construction 6.1.3.20]{HA}) and the fact that $\Sp$ is stable that $\cross_*$ preserves colimits. 
    As $\mlin$ is a left adjoint, we can conclude that $\partial_*$ preserves colimits.

    To see that the functor also preserves finite limits, note that the forgetful functor
    \[
    {\LMod}_{\partial_*{\id_\sD}}(\SSeqnu(\Sp)) \to \SSeqnu(\Sp)
    \]
    creates all limits by \cite[Corollary 4.2.3.3]{HA},
    so that it suffices to show that
    \[
        \partial_* \colon \Fun^{\ast, \omega}(\sC, \sD) \to \SSeqnu(\Sp)
    \]
    preserves finite limits.
    This again follows by writing $\partial_*$ as multilinearized cross-effects: the functor $\cross_*$ is a right adjoint and therefore commutes with all limits, whereas $\mlin$ preserves finite limits by Goodwillie's formula for multilinearization.
\end{proof}

We will also need the following:

\begin{proposition} \label{prop: derivative-left-adjoint}
    Let $F \colon \sC \to \sD$ be a left adjoint between differentiable categories.
    Then there is an equivalence
    \[
    \partial_*{F} \simeq \partial_*{\id_\sD} \circ \partial_1{F}
    \]
    of left $\partial_*{\id_\sD}$-modules.
    In other words, $\partial_*{F}$ is the free left $\partial_*{\id_\sD}$-module on the symmetric sequence $\partial_1{F}$.
\end{proposition}
\begin{proof}
    Since $F$ preserves colimits, the functor $\Sigma^\infty_\sD \circ F \colon \sC \to \Sp(\sD)$ is linear.
    It follows that $\partial_*(\Sigma^\infty_\sD \circ F)$ is concentrated in arity $1$, where it is given by $\partial_1(F)$.
    This symmetric sequence is a trivial left $\partial_*(\Sigma^\infty_\sD \Omega^\infty_\sD)$-comodule for degree reasons, so that its Koszul dual is the free left $\partial_*{\id_\sD}$-module on $\partial_1{F}$.
    But this Koszul dual is equivalent to $\partial_*{F}$ by \cref{prop: derivatives-stabilization-indec}, finishing the proof.
\end{proof}

\begin{remark} \label{rem: derivative-right-adjoint}
    One can prove in the same way that for a right adjoint $G \colon \sD \to \sC$, we have $\partial_*{G} \simeq \partial_1{G} \circ \partial_*{\id_\sD}$.
    In other words, $\partial_*{G}$ is the free \emph{right} $\partial_*{\id_\sD}$-module on $\partial_1{G}$.
\end{remark}

\begin{remark} \label{rem: derivative-colim-pres-multivar-functor}
    The same proof also shows that if $F \colon \sC^{\times n} \to \sD$ is a functor that preserves colimits in each variable separately, then $\partial_*F$ is a free left $\partial_*{\id_\sD}$-module on $\partial_{(1, \ldots, 1)}(F)$, the derivative of $F$ in multidegree $(1, \ldots, 1)$. 
    See \cref{sec: multivariate-calculus} below for more on derivatives in multiple variables.
\end{remark}

\subsection{Differentiating symmetric monoidal \texorpdfstring{$\infty$}{infinity}-categories}

We will now show that the assignment that sends $\sC$ to either of the $\infty$-categories $\Alg_{\partial_*{\id_\sC}}(\Sp)$ or $\LMod_{\partial_*{\id_\sC}}(\SSeqnu(\Sp))$ can be turned into a functor of $(\infty, 2)$-categories.
We will also explain that this construction carries a symmetric monoidal structure on $\sC$ to one on $\Alg_{\partial_*{\id_\sC}}(\Sp)$ and $\LMod_{\partial_*{\id_\sC}}(\SSeqnu(\Sp))$, and preserves (lax) symmetric monoidal functors.

Up to this point, we have assumed that our differentiable $\infty$-categories $\sC$ come equipped with an equivalence $\Sp(\sC) \simeq \Sp$.
This assumption was just a convenience, but not really necessary.
Given a reduced finitary functor $F \colon \sC \to \sD$ between arbitrary differentiable $\infty$-categories, its derivatives $\partial_*F$ form an object in the category $\SSeqnu(\Sp(\sC), \Sp(\sD))$ of \emph{nonunital functor symmetric sequences} from $\Sp(\sC)$ to $\Sp(\sD)$, which is a generalization of the ordinary notion of symmetric sequence.
We will not recall the precise definition here, but suffice it to say that a composition product for functor symmetric sequences can be defined, and all results stated in the previous section hold in this generality as well; see \cite{blansblom2025chainrulegoodwilliecalculus} for details\footnote{The category $\SSeqnu(\sC, \sD)$ was denoted by $\SymFunL_{\geq 1}(\sC, \sD)$ in \cite{blansblom2025chainrulegoodwilliecalculus}.}.
In particular, there is an $(\infty, 2)$-precategory $\MMor$ whose objects are pairs $(\sC, \sO)$ where $\sC$ is a stable category and $\sO$ is an algebra in $\SSeqnu(\sC, \sC)$.
The $\infty$-category $\SSeqnu(\sC, \sD)$ is compatibly left-tensored by $\SSeqnu(\sC, \sC)$ and right-tensored by $\SSeqnu(\sD, \sD)$, and morphisms from $(\sC, \sO)$ to $(\sD, \sP)$ in $\MMor$ are given by $(\sP, \sO)$-bimodules in $\SSeqnu(\sC, \sD)$.
Writing $\diff$ for the $(\infty, 2)$-category of differentiable $\infty$-categories and reduced finitary functors, there is a functor
\[
\partial_* \colon \diff \to \MMor
\]
that sends $\sC$ to the pair $(\Sp(\sC), \partial_*{\id_\sC})$ and $F \colon \sC \to \sD$ to the $(\partial_*{\id_\sD}, \partial_*{\id_\sC})$-bimodule $\partial_*F$ \cite[Theorem 4.4.7]{blansblom2025chainrulegoodwilliecalculus}. 

For any pair of stable $\infty$-categories $\sC$ and $\sD$, there is an evaluation functor
\[
- \circ - \colon \SSeqnu(\sC, \sD) \times \sC \to \sD,
\]
In case $\sC = \sD$, this is part of an action of the category $\SSeq(\sC, \sC)$ on $\sC$.
If $\sC$ and $\sD$ are both equivalent to the category of spectra, this action is the one corresponding to the monoidal functor $\Sym \colon \SSeq(\Sp) \to \End(\Sp)$.
If $\sO \in \Alg(\SSeq(\sC, \sC))$, we use this action to define the category $\Alg_\sO(\sC)$ of $\sO$-algebras in $\sC$, exactly as we did for operads in spectra.
The $(\infty, 2)$-precategory $\MMor$ has a final object $\ast$, which has the property that for any $(\sC, \sO) \in \MMor$, we have an equivalence 
\[
\MMor(\ast, (\sC, \sO)) \simeq \Alg_\sO(\sC),
\]
where the left hand side denotes the mapping category in $\MMor$.

\begin{definition}
We write
\[
\reAlg \colon \MMor \to \diff, \qquad \reLMod \colon \MMor \to \diff
\]
for the functors corepresented by $\ast$ and $(\Sp, \unit)$ respectively, where $\unit$ denotes the trivial operad in spectra.  
We define the functors
\[
    \maAlg \colon \diff \to \diff \quad \text{and} \quad \maLMod \colon \diff \to \diff
\]
as the composites $\reAlg \circ \partial_*$ and $\reLMod \circ \partial_*$ respectively.
\end{definition}

We will often refer to either of the functors $\maAlg$ or $\maLMod$ as the \emph{Goodwillie transform}. Note that for an operad in spectra $\sO$, we have equivalences
\[
\reAlg(\sO) \simeq \Alg_\sO(\Sp), \qquad \reLMod(\sO) \simeq \LMod_\sO(\SSeqnu(\Sp)).
\]
Moreover, $\reAlg$ sends a $(\sP, \sO)$-bimodule $M$ to the functor
\[
M \circ_\sO - \colon \Alg_\sO(\Sp) \to \Alg_\sP(\Sp),
\]
given by taking an $\sO$-algebra $A$ to the two-sided bar construction $M \circ_\sO A$.
Similarly, $\reLMod$ sends $M$ to $M \circ_\sO -$ considered as a functor on left modules.
The functor $\maAlg$ therefore sends $\sC \in \diff_\Sp$ to $\Alg_{\partial_*{\id_\sC}}(\Sp)$, and it sends $F \colon \sC \to \sD$ to the functor
\[
\partial_*F \circ_{\partial_*{\id_\sC}} - \colon \Alg_{\partial_*{\id_\sC}}(\Sp) \to \Alg_{\partial_*{\id_\sD}}(\Sp).
\]
The functor $\maLMod$ has a similar description.

The functors $\reAlg$ and $\reLMod$ have the following exactness properties.

\begin{proposition} \label{prop: materialization-exactness}
    The functors
    \[
    \reAlg \colon \MMor \to \diff, \qquad \reLMod \colon \MMor \to \diff
    \]
    preserve colimits and finite limits on mapping categories.
\end{proposition}
\begin{proof}
    We will give the proof for $\reAlg$; the proof for $\reLMod$ is exactly the same.
    We will also only show this for mapping categories between operads in spectra, since this suffices for our applications.
    (The general case is proved in exactly the same way.)

    Suppose that $\sO, \sP \in \Op(\Sp)$.
    We need to show that the functor
    \begin{align*}
        \BMod_{(\sP, \sO)}(\SSeqnu(\Sp)) &\to \Fun(\Alg_\sO(\Sp), \Alg_\sP(\Sp)) \\
        M &\mapsto M \circ_{\sO} -
    \end{align*}
    preserves colimits and finite limits.
    For colimits, this follows since the composition product preserves colimits in the first variable.
    For finite limits, this follows since limits and sifted colimits in $\Alg_{\sP}(\Sp)$ are computed on underlying objects, and the composition product commutes with finite limits in the first variable.
\end{proof}

\begin{corollary}
\label{cor:exactnessGoodwillietransform}
    The functors 
    \[
    \maAlg \colon \diff \to \diff \quad \text{and} \quad \maLMod \colon \diff \to \diff
    \]
    preserve colimits and finite limits on mapping categories.
\end{corollary}
\begin{proof}
    This follows by combining the previous proposition with \cref{prop: derivatives-lmod-colim}.
\end{proof}

The next proposition shows that a bimodule can be recovered from its value under the functor $\reAlg \colon \MMor \to \diff$.

\begin{proposition} \label{prop: derivatives-bimodule-gives-bimodule}
    Let $\sC$ and $\sD$ be stable presentable categories and let 
    \[
    \sO \in \Alg(\SSeqnu(\sC, \sC)), \qquad \sP \in \Alg(\SSeqnu(\sD, \sD)).
    \]
    Suppose that $M \in \BMod_{(\sP, \sO)}(\SSeqnu(\sC, \sD))$.
    Then the derivatives of the functor
    \[
    M \circ_\sO - \colon \Alg_\sO(\sC) \to \Alg_\sP(\sD)
    \]
    are given by the $(\sP, \sO)$-bimodule $M$.
\end{proposition}
\begin{proof}
We will prove this in the case $\sC = \Sp = \sD$; the general proof again proceeds in exactly the same way.
Note that there are the following equivalences of functors from $\Alg_\sO(\Sp)$ to $\Alg_\sP(\Sp)$: 
\[
M \circ_\sO - \simeq \colim_{[n] \in \Delta^\op} M \circ \sO^{\circ n} \circ - \simeq \colim_{[n] \in \Delta^\op} \Sym_{M \circ \sO^{\circ n}} \circ \forget_{\sO}.
\]
Taking derivatives, we find that
\[
    \partial_*(M \circ_\sO -) \simeq \colim_{[n] \in \Delta^\op} \partial_*\Sym_{M \circ \sO^{\circ n}} \circ \partial_* \forget_\sO \simeq \colim_{[n] \in \Delta^{\op}} M \circ \sO^{\circ n} \circ \sO \simeq M
\]
The first equivalence follows from the derivatives functor preserving colimits when taking values in bimodules (\cref{prop: derivatives-lmod-colim}) and the stable chain rule (\cref{thm: lax-derivatives}).
The second equivalence follows from $\partial_*$ being a left inverse of $\Sym$ (\cref{ex: der-retraction-sym}) and our computation of the derivatives of right adjoints (\cref{rem: derivative-right-adjoint}).
The final equivalence follows since the simplicial diagram has an extra degeneracy.
\end{proof}

\begin{remark}
    This proposition suggests that there is an equivalence $\partial_* \circ \reAlg \simeq \id_{\MMor}$ of functors of $(\infty, 2)$-precategories.
    We indeed expect this to be true, and will return to this in future work.
\end{remark}

The following less obvious fact will be proved in \cite{BlansBlom2026}.
\begin{proposition} \label{prop: partial-mor-preserves-products}
    The functor $\partial_* \colon \diff \to \MMor$ preserves finite products.
\end{proposition}

\begin{corollary}
    The functors $\maAlg$ and $\maLMod \colon \diff \to \diff$ preserve finite products.
\end{corollary}
\begin{proof}
    This holds since $\maAlg = \reAlg \circ \partial_*$ and $\maLMod = \reLMod \circ \partial_*$, and $\reAlg$ and $\reLMod$ are corepresentable functors.
\end{proof}

This result allows us to differentiate symmetric monoidal structures.

\begin{proposition} \label{prop: differentiating-sym-mon-cats}
    Let $\sC \in \diff_\Sp$ and suppose $\sC$ is equipped with a nonunital symmetric monoidal structure such that the tensor product $-\otimes_\sC-$ preserves filtered colimits and satisfies $\ast \otimes_\sC \ast \simeq \ast$, where $\ast$ denotes the zero object of $\sC$.
    Then $\Alg_{\partial_*{\id_\sC}}(\Sp)$ admits a nonunital symmetric monoidal structure with tensor product
    \[
    \partial_*(- \otimes_\sC -) \circ_{\partial_*{\id_\sC} \times \partial_*{\id_\sC}} - \colon \Alg_{\partial_*{\id_\sC}}(\Sp) \times \Alg_{\partial_*{\id_\sC}}(\Sp) \to \Alg_{\partial_*{\id_\sC}}(\Sp).
    \]
    If $\sD$ is another differentiable symmetric monoidal category satisfying these properties, and $F \colon \sC \to \sD$ is a (lax) symmetric monoidal functor, then the functor
    \[
    \partial_*F \circ_{\partial_*{\id_\sC}} - \colon \Alg_{\partial_*{\id_\sC}}(\Sp) \to \Alg_{\partial_*{\id_\sD}}(\Sp) 
    \]
    inherits a (lax) symmetric monoidal structure.
    The same result holds with $\Alg_{\partial_*{\id_\sC}}(\Sp)$ replaced by $\LMod_{\partial_*{\id_\sC}}(\SSeqnu(\Sp))$.
\end{proposition}
\begin{proof}
Observe that a nonunital symmetric monoidal category $\sC$ as in the statement of the proposition is the same as a Segal object in the $(\infty, 2)$-category $\diff$ indexed by the category $\Surj_*$ of pointed finite sets and surjections.
Since $\maAlg$ and $\maLMod$ preserve finite products, they preserve such Segal objects.
As a (lax) symmetric monoidal functor between such nonunital symmetric monoidal categories can be expressed as a (partially lax) natural transformation between such Segal objects by for instance \cite[Theorem E]{haugseng-hebestreit-linskens-nuiten-lax-monoidal-adjunctions}, and since such natural transformations are preserved by all functors of $(\infty, 2)$-categories, it follows that $\maAlg$ and $\maLMod$ preserve (lax) symmetric monoidal functors as well.
\end{proof}

\begin{remark}
    One could try to obtain a version of the previous proposition for \emph{unital} symmetric monoidal differentiable categories $\sC$ by working with Segal objects in $\diff$ indexed by $\Fin_*$ instead of $\Surj_*$.
    But then the functor $\ast \to \sC$ picking out the unit object should be reduced, as it has to be a morphism in $\diff$.
    This forces the unit of $\sC$ to be the zero object.
    If this hypothesis is satisfied, it follows that we get an induced \emph{unital} symmetric monoidal structure on $\Alg_{\partial_*{\id_\sC}}(\Sp)$ with unit given by the zero object.
\end{remark}

\begin{example} \label{ex: derivative-cartesian-symmetric-monoidal-structure}
    Let $\sC \in \diff_\Sp$ and suppose we give $\sC$ the cartesian symmetric monoidal structure.
    Then the induced symmetric monoidal structures on both $\Alg_{\partial_*{\id_\sC}}(\Sp)$ and $\LMod_{\partial_*{\id_\sC}}(\SSeqnu(\Sp))$ are again cartesian, as follows easily from \cref{prop: materialization-exactness}.
    The same holds for cocartesian symmetric monoidal structures.
\end{example}

\begin{example} \label{ex: derivative-smash-product-sym-mon}
    Let $\sC \in \diff_\Sp$.
    The smash product of $X, Y \in \sC$ is defined to be the cofiber of the canonical map $X \sqcup Y \to X \times Y$. 
    Suppose that this extends to a nonunital symmetric monoidal structure on $\sC$; conditions under which this holds are given in \cref{prop: smash-monoidal-structure}.
    It then follows from the previous example combined with \cref{prop: materialization-exactness} that the induced nonunital symmetric monoidal structures on both $\Alg_{\partial_*{\id_\sC}}(\Sp)$ and $\LMod_{\partial_*{\id_\sC}}(\SSeqnu(\Sp))$ are again given by the smash product.
\end{example}

\begin{example} \label{ex: derivative-tensor-product-sym-mon}
    Consider $\Sp$ as a symmetric monoidal category under the tensor product.
    The functor $- \otimes - \colon \Sp \times \Sp \to \Sp$ is bilinear, so that the bisymmetric sequence $\partial_*(- \otimes -)$ is concentrated in bidegree $(1, 1)$, where it is given by the sphere spectrum.
    It follows that the induced nonunital symmetric monoidal structure on $\Alg_{\partial_*{\id_\Sp}}(\Sp) \simeq \Sp$ is again the tensor product.
    The induced nonunital symmetric monoidal structure on $\SSeqnu(\Sp)$ is Day convolution.
    Similarly, the augmented tensor product $\otimesaug$ on $\Sp$ gets sent to itself under $\maAlg$ and to the augmented Day convolution tensor product under $\maLMod$.
\end{example}

\subsection{Goodwillie calculus in two variables} \label{sec: multivariate-calculus}

We make some use of Goodwillie calculus in two variables throughout this article.
The goal of the present section is to explain what form the derivatives of two-variable functors take by giving an explicit description of the category of functor symmetric sequences $\SSeqnu(\Sp \times \Sp, \Sp)$. 

\begin{definition}
We define the category of \emph{nonunital bisymmetric sequences} $\biSSeq(\Sp)$ as the full subcategory of
\[
\Fun(\Fin^{\cong} \times \Fin^{\cong}, \Sp)
\]
spanned by those functors $A$ that send the pair $(\emptyset, \emptyset)$ to the zero object.
\end{definition}

In other words, to give a bisymmetric sequence $A$ is to give for each pair $(n, m)$ of integers for which $n + m \geq 1$ a spectrum $A_{n,m}$ with $\Sigma_n \times \Sigma_m$-action.

\begin{proposition}
    There is an equivalence of categories
    \[
    \SSeqnu(\Sp \times \Sp, \Sp) \simeq \biSSeq(\Sp).
    \]
\end{proposition}
\begin{proof}
By \cite[Section 3.1]{blansblom2025chainrulegoodwilliecalculus} the category $\SSeqnu(\Sp \times \Sp, \Sp)$ is equivalent to 
\[
    \Fun^{\mathrm{L}}(\bigoplus_{n \geq 1} (\Sp \oplus \Sp)^{\otimes n}_{h\Sigma_n}, \Sp),
\] 
where $\oplus$ and $\otimes$ denote the coproduct and tensor product in $\presl$ respectively, and $\FunL$ denotes the category of colimit preserving functors.
We then have the chain of equivalences
\begin{align*}
    \FunL(\bigoplus_{n \geq 1} (\Sp \oplus \Sp)^{\otimes n}_{h\Sigma_n}, \Sp) & \simeq \FunL(\bigoplus_{m + n \geq 1} \Sp^{\otimes m}_{h\Sigma_m} \otimes \Sp^{\otimes n}_{h\Sigma_n}, \Sp) \\
    &\simeq \prod_{m + n \geq 1} \FunL(\Sp_{h \Sigma_m \times \Sigma_n}, \Sp) \\
    &\simeq \prod_{m + n \geq 1} \Fun(B\Sigma_m \times B\Sigma_n, \Sp) \\
    &\simeq \biSSeq(\Sp). \qedhere
\end{align*}
\end{proof}

\begin{remark}
  It follows from this proposition that if $\sC$ and $\sD$ are differentiable categories with specified equivalences $\Sp(\sC) \simeq \Sp \simeq \Sp(\sD)$, and $F \colon \sC \times \sC \to \sD$ is a reduced finitary functor, then the derivatives $\partial_*F$ form a bisymmetric sequence in $\Sp$.
  The $(m, n)$-derivative $\partial_{m,n}F$ can be computed by applying the composite functor
    \begin{align*}
        \Fun^{\ast, \omega}(\sC \times \sC, \sD) & \xrightarrow{\phantom{lll}\sim\phantom{lll}} \Fun^{\ast, \omega}(\sC, \Fun^{\ast, \omega}(\sC, \sD)) \\
        & \xrightarrow{(\partial_n)_* \phantom{|}} \Fun^{\ast, \omega}(\sC, \Fun(B \Sigma_n, \Sp)) \\
        & \xrightarrow{\phantom{(}\partial_m\phantom{)_*}} \Fun(B\Sigma_m, \Fun(B \Sigma_n, \Sp)) \\
        & \xrightarrow{\phantom{lll}\sim\phantom{lll}} \Fun(B\Sigma_m \times B \Sigma_n, \Sp)
    \end{align*}
    to $F$.
    One obtains the same result by first taking the $m$-th and then the $n$-th derivative.
\end{remark}

\begin{remark}
To $A \in \biSSeq(\Sp)$ we can attach a functor $\Sym_A \colon \Sp \times \Sp \to \Sp$ given by the formula
\[
\Sym_A(X, Y) = \bigoplus_{n + m \geq 1} (A_{n, m} \otimes X^{\otimes n} \otimes Y^{\otimes m})_{h\Sigma_n \times \Sigma_m}.
\]
This defines a functor $\Sym \colon \biSSeq(\Sp) \to \Fun(\Sp \times \Sp, \Sp)$.
The evaluation functor
\[
 - \circ - \colon \biSSeq(\Sp) \times \Sp \times \Sp \to \Sp
\]
is given by $A \circ (X, Y) = \Sym_A(X, Y)$.
\end{remark}

\begin{remark}
    If $F \colon \Spc_* \times \Spc_* \to \Spc_*$ is a reduced finitary functor, then $\partial_*F \in \biSSeq(\Sp)$ has the structure of an $(\LL, \LL \times \LL)$-bimodule, where $\LL \times \LL$ denotes the cartesian product in the Morita category.
    We will denote the category of such bimodules by 
    \[
    \BMod_{\LL}(\biSSeq(\Sp)).
    \]
    These bimodules are defined using the composition product on the category of functor symmetric sequences $\SSeqnu(\Sp^{\times 2}, \Sp^{\times 2})$ as well as the left and right tensoring of $\biSSeq(\Sp)$ by $\SSeqnu(\Sp)$ and $\SSeqnu(\Sp^{\times 2}, \Sp^{\times 2})$ respectively.
    Explicit formulas for this composition product and these tensorings can be deduced from the general formula in \cite[Proposition 3.2.9]{blansblom2025chainrulegoodwilliecalculus},
    but we do not need them here: all our computations of bimodules of the form $\partial_*F$ will follow from the general exactness properties of $\partial_*$ proved in the previous sections.
\end{remark}

\begin{remark}
    Everything written in this section generalizes to functors in more than two variables.
    Concretely, the category of functor symmetric sequences $\mSSeqnu{n}$ for $n \geq 1$ can be identified with the category of $n$-fold symmetric sequences: to give an object $A$ in this category is to give for each $n$-tuple $(m_1, \ldots, m_n)$ with not all $m_i$ equal to zero a spectrum $A_{m_1, \ldots, m_n}$ with $\Sigma_{m_1} \times \cdots \times \Sigma_{m_n}$-action.
    Such an $n$-fold symmetric sequence gives rise to a functor $\Sym_A \colon \Sp^{\times n} \to \Sp$, given by the formula
    \[
    \Sym_A(X_1, \ldots, X_n) = \bigoplus_{m_1 + \cdots + m_n \geq 1} (A_{m_1, \ldots, m_n} \otimes X_1^{\otimes m_1} \otimes \cdots \otimes X_n^{\otimes m_n})_{h\Sigma_{m_1} \times \cdots \times \Sigma_{m_n}}.
    \]
\end{remark}

We will need the following simple computation of a composition product of multisymmetric sequences in \cref{sec: free-Lie-algebra}.

\begin{proposition} \label{prop: multi-composition-product}
    Suppose we are given integers $n, m_1, \ldots m_n \geq 1$ and symmetric sequences $A \in \mSSeqnu{n}$ and $B^i \in \mSSeqnu{m_i}$ for $1 \leq i \leq n$.
    Assume that both $A$ and the $B^i$ are concentrated in multidegree $(1, \ldots, 1)$, where for $A$ there are $n$ occurrences of $1$, and for $B^i$ there are $m_i$ occurrences of $1$.
    Then the composition product $A \circ (B^1, \ldots B^n)$ is concentrated in multidegree $(1, \ldots, 1)$, where there are $m_1 + \cdots + m_n$ occurrences of $1$.
\end{proposition}
\begin{proof}
    It is possible to deduce this from the general formula for the composition product given in \cite[Proposition 3.2.9]{blansblom2025chainrulegoodwilliecalculus}.
    This is however not the easiest way to proceed, since these formulas quickly become unwieldy.
    Instead, we make use of the fact that the assignment $X \mapsto \Sym_X$ sends the composition product to functor composition.
    By assumption, we have
    \begin{align*}
        \Sym_A(X_1, \ldots, X_n)  &= A_{1, \ldots, 1} \otimes X_1 \otimes \cdots \otimes X_n, \\
        \Sym_{B^i}(Y_1, \ldots, Y_{m_i}) &= B^i_{1, \ldots, 1} \otimes Y_1 \otimes \cdots \otimes Y_{m_i} \quad \text{for $1 \leq i \leq n$.}
    \end{align*}
    Composing and relabeling variables, we find that $\Sym_{A \circ (B^1, \ldots, B^n)}$ is given by the expression
    \[
    A_{(1, \ldots, 1)} \otimes B^1_{(1, \ldots, 1)} \otimes \cdots \otimes B^n_{(1, \ldots, 1)} \otimes Y_1 \otimes \cdots \otimes Y_{m_1 + \cdots + m_n}.
    \]
    It follows that $A \circ (B^1, \ldots, B^n)$ is concentrated in multidegree $(1, \ldots, 1)$, where it is given by the spectrum $A_{(1, \ldots, 1)} \otimes B^1_{(1, \ldots, 1)} \otimes \cdots \otimes B^n_{(1, \ldots, 1)}$.
\end{proof}

\section{The homotopy theory of Lie algebras} \label{sec: homotopy-theory-of-lie-algebras}
The goal of this section is to use Goodwillie calculus to transfer a number of classical results about pointed spaces to the homotopy theory of spectral Lie algebras.

\subsection{The commutator cofiber sequence}

We will now prove \cref{thm: mainA}, showing that there is a natural cofiber sequence
\[
    \Sigma (X \wedge Y) \to \Sigma X \vee \Sigma Y \to \Sigma X \times \Sigma Y,
\]
of spectral Lie algebras which desuspends the cofiber sequence defining the smash product of $\Sigma X$ and $\Sigma Y$.
We start by recalling how the maps in this sequence and a nullhomotopy of their composite are constructed.

Let $\sC$ be a pointed $\infty$-category with finite limits and colimits.
For $X, Y \in \sC$, we write $X \vee Y$ for their coproduct and $X \wedge Y$ for their smash product, i.e.\ the cofiber of the canonical map $X \vee Y \to X \times Y$.
For $X \in \sC$, the object $\Omega X$ is an $\E_1$-group.
The composite
\[
\Omega X \vee \Omega X \to \Omega X \times \Omega X \xrightarrow{((xy)x^{-1})y^{-1}} \Omega X
\]
is canonically nullhomotopic, and hence factors over the smash product to give the commutator map
\[
c \colon \Omega X \wedge \Omega X \to \Omega X.
\]

\begin{definition}[Whitehead bracket]\label{def: whitehead-bracket}
Let $\sC$ be a pointed $\infty$-category with finite limits and colimits. 
Suppose we are given maps $f \colon \Sigma A \to X$ and $g \colon \Sigma B \to X$ in $\sC$.
The \emph{Whitehead bracket} of $f$ and $g$ is the map
\[
[f, g] \colon \Sigma(A \wedge B) \to X
\] 
defined as the adjoint of the composite
\[
\begin{tikzcd}
    A \wedge B \ar[r, "\bar{f} \wedge \bar{g}"] & \Omega X \wedge \Omega X \ar[r, "c"] & \Omega X,
\end{tikzcd}
\]
where $\bar{f}$ and $\bar{g}$ denote the adjoints of $f$ and $g$ respectively.
\end{definition}

    \begin{construction}[Commutator sequence]
         Let $\sC$ be a pointed $\infty$-category with finite limits and colimits.
         For $X, Y \in \sC$, we obtain a sequence
        \[
		\begin{tikzcd}
			\Sigma(X \wedge Y) \ar[r, "{[\iota_X, \iota_Y]}"] & \Sigma X \vee \Sigma Y \ar[r] & \Sigma X \times \Sigma Y,
		\end{tikzcd}
	    \]
        where the first map is the Whitehead bracket of the inclusions of $\Sigma  X$ and $\Sigma Y$ in $\Sigma X \vee \Sigma Y$, and the second map is the canonical one from the coproduct to the product.
        The composite is canonically nullhomotopic, as can easily be seen by considering its adjoint.
        We call this null-sequence the \emph{commutator sequence} of $X$ and $Y$.
        It can be made natural in $X$ and $Y$ by considering the commutator sequence of the coordinate projections in $\Fun(\sC \times \sC, \sC)$.
    \end{construction}

    \begin{remark}
        The Whitehead bracket and the commutator sequence are preserved by any functor $F \colon \sC \to \sD$ that preserves finite limits and colimits.
    \end{remark}

    We can now state the main result of this section.

    \begin{theorem}[\cref{thm: mainA}] \label{thm: Lie-commutator-cofiber-seq}
    The commutator sequence of any pair of spectral Lie algebras $X$ and $Y$ is a cofiber sequence in $\Lie(\Sp)$.
    In case $X$ and $Y$ are free Lie algebras $\free_{\LL}(A)$ and $\free_{\LL}(B)$ respectively, this cofiber sequence is equivalent to
    \[
    \free_{\LL}(\Sigma^{-1} A \otimes B) \xrightarrow{[\cdot,\cdot]} \free_{\LL}(A \oplus B) \to \free_{\LL}(A) \times \free_{\LL}(B),
    \]
    where the map denoted $[\cdot, \cdot]$ is the map of spectral Lie algebras induced by
    \[
    [\iota_{X}, \iota_{Y}] \colon \Sigma^{-1} X \otimes Y \to \free_{\LL}(X \oplus Y),
    \]
    the Lie bracket of the inclusions of $X$ and $Y$ into $\free_{\LL}(X \oplus Y)$.    
    \end{theorem}
    
    We will need three lemmas to prove this theorem.
	We start with the following classical fact:
	\begin{lemma} \label{lem: spaces-comm-cofiber-seq}
		The commutator sequence in $\Fun^{\ast, \omega}(\Spc_* \times \Spc_*, \Spc_*)$ is a cofiber sequence.
	\end{lemma}
	\begin{proof}[Sketch of proof]
        First note that all functors in the commutator sequence indeed lie in the full subcategory of reduced finitary functors.
        Since colimits in a functor category are computed pointwise, it suffices to show that the commutator sequence of any pair $X, Y \in \Spc_*$ is a cofiber sequence.
		Here one can argue reducing to the case where $X$ and $Y$ are both wedges of spheres, since all functors in the sequence commute with weakly contractible colimits in both variables.
		In this case, it reduces to the fact that the attaching map of the top cell in a product of spheres $S^k \times S^\ell$ is given by the Whitehead bracket.
	\end{proof}

    We will use the following two lemmas to show the cofiber sequence is of the right form for free Lie algebras.
    First of all, the free Lie algebra functor takes tensor products to smash products.

\begin{lemma} \label{lem: derivatives-smash-product}
    Let $- \wedge - \colon \Spc_* \times \Spc_* \to \Spc_*$ be the smash product functor.
    Then we have an equivalence of spectral Lie algebras
    \[
    \free_{\LL}(X) \wedge \free_{\LL}(Y) \simeq \free_{\LL}(X \otimes Y),
    \]
    natural in $X, Y \in \Sp$.
\end{lemma}

\begin{proof}
    Since $\Sigma^\infty$ is symmetric monoidal from the smash product of pointed spaces to the tensor product of spectra, we have a natural isomorphism
    \[
    \Sigma^\infty \circ (- \wedge -) \simeq (- \otimes -) \circ (\Sigma^\infty \times \Sigma^\infty)
    \]
    of functors $\Spc_* \times \Spc_* \to \Sp$.
    The stable chain rule then implies that
    \[
    \partial_*(\Sigma^\infty \circ (- \wedge -)) \simeq \partial_*(- \otimes -) \circ (\partial_*\Sigma^\infty \times \partial_*\Sigma^\infty) \simeq \partial_*(- \otimes -).
    \]
    The bisymmetric sequence $\partial_*(- \otimes -)$ is concentrated in bidegree $(1,1)$, so that the left $\partial_*(\Sigma^\infty \Omega^\infty)$-comodule structure is trivial for degree reasons.
    Taking Koszul duals and using \cref{prop: derivatives-stabilization-indec},  we find that
    \[
    \partial_*(- \wedge -) \simeq \LL \circ \partial_*(- \otimes -)
    \]
    as left $\LL$-modules.
    We now let both sides of this equation act on a pair of spectra $(X, Y)$.
    The left hand side then becomes
    \begin{align*}
    \partial_*(- \wedge -) \circ (X, Y) 
    &\simeq \partial_*(- \wedge -) \circ_{\LL \times \LL} (\free_\LL(X), \free_\LL(Y)) \\
    &\simeq \free_{\LL}(X) \wedge \free_\LL(Y).
    \end{align*}
    Here the final equivalence follows from \cref{prop: materialization-exactness}.
    Since $\partial_*(- \otimes -) \circ (X, Y) = X \otimes Y$, the right hand side evaluates to $\free_{\LL}(X \otimes Y)$.
\end{proof}    

Finally, we need to relate the Whitehead bracket to the Lie bracket of spectral Lie algebras.
More precisely, we let
\[
B_W \colon \free_{\LL}(\Sigma A \otimes B) \simeq \Sigma \free_{\LL}(A) \wedge \free_{\LL}(B) \to  \free_\LL(\Sigma A \oplus \Sigma B)
\]
denote the composite of the equivalence from the previous lemma and the Whitehead bracket.
We write
\[
B_L \colon \free_{\LL}(\Sigma A \otimes B) \to \free_{\LL}(\Sigma A \oplus \Sigma B)
\]
for the Lie bracket, which is the map of Lie algebras determined by the composite
\[
\Sigma A \otimes B \simeq \Sigma^{-1} \Sigma A \otimes \Sigma B\to \Sigma^{-1}\free_{\LL}(\Sigma A \oplus \Sigma B)^{\otimes 2} \xrightarrow{\mu} \free_{\LL}(\Sigma A \oplus \Sigma B).
\]
Here, the first equivalence interchanges the suspension with $A$ and uses $\id \simeq \Sigma^{-1}\Sigma$. 
The second map is the desuspension of the composite of the inclusions of $\Sigma A$ and $\Sigma B$ in $\Sigma A \oplus \Sigma B$ and the unit $\Sigma A \oplus \Sigma B \to \free_{\LL}(\Sigma A \oplus \Sigma B)$.
The third map is the multiplication map coming from the Lie algebra structure on $\free_{\LL}(\Sigma A \oplus \Sigma B)$, where we have made the identification $\LL(2) \simeq \Sph^{-1}$.

\begin{lemma} \label{lem: lie-bracket-is-whitehead}
    The natural transformations $B_W$ and $B_L$ agree, possibly up to a sign.
\end{lemma}
\begin{proof}
    It suffices to show that the composites of $B_W$ and $B_L$ with the inclusion $\Sigma A \otimes B \to \free_{\LL}(\Sigma A \otimes B)$ are equivalent natural transformations of functors $\Sp \times \Sp \to \Sp$.
    We write $B_W$ and $B_L$ for these composites as well, and claim that both of them exhibit $\Sigma A \otimes B$ as the $(1,1)$-derivative of $\free_{\LL}(\Sigma A \oplus \Sigma B)$.
    This is enough to complete the proof, since it implies that $B_W$ and $B_L$ agree up to a natural automorphism of the functor $(A, B) \mapsto \Sigma A \otimes B$, of which there are exactly two: the identity and multiplication by $-1$.
    
    For $B_L$, the claim follows by observing that it is given by the composite
    \[
        \Sigma A \otimes B \to \LL(2) \otimes_{h\Sigma_2} (\Sigma A \oplus \Sigma B)^{\otimes 2} \to \free_{\LL}(\Sigma A \oplus \Sigma B),
    \]
    where the first map uses the inclusions of $\Sigma A$ and $\Sigma B$ in $\Sigma A \oplus \Sigma B$ as well as the equivalence $\LL(2) \simeq \Sph^{-1}$ and the second map is the inclusion of the weight 2 term.
    For $B_W$, observe that since it is given by the composite
    \[
    \Sigma A \otimes B \to \free_{\LL}(\Sigma A \otimes B) \simeq \Sigma \free_{\LL}(A) \wedge \free_{\LL}(B) \to \free_{\LL}(\Sigma A \oplus \Sigma B),
    \]
    where the first map clearly exhibits the source as the $(1,1)$-derivative of $\free_{\LL}(\Sigma A \oplus \Sigma B)$, it suffices to show that the Whitehead bracket
    \[
        \Sigma \free_{\LL}(A) \wedge \free_{\LL}(B) \to \free_{\LL}(\Sigma A \oplus \Sigma B)
    \]
    induces an equivalence on $(1,1)$-derivatives.
    But this map is induced by taking derivatives of the Whitehead bracket
    \[
    \Sigma X \wedge Y \to \Sigma X \vee \Sigma Y
    \]
    in pointed spaces, so it suffices to show that this map induces an equivalence of $(1,1)$-derivatives by \cref{prop: derivatives-bimodule-gives-bimodule}.
    This is a consequence of the Hilton-Milnor theorem; see Lemma 2.4 of \cite{brantnerheuts}.
 \end{proof}

 \begin{remark}
     A version of this lemma also appears as \cite[Proposition 4.28]{heutsLie}. However, the proof given there is incomplete.
 \end{remark}

    We now come to the proof of the main theorem of this section.

\begin{proof}[Proof of \cref{thm: Lie-commutator-cofiber-seq}]
	Since the Goodwillie transform 
	\[
	\maAlg \colon \Fun^{\ast, \omega}(\Spc_* \times \Spc_*, \Spc_*) \to \Fun^{\ast, \omega}(\Lie(\Sp) \times \Lie(\Sp), \Lie(\Sp))
	\]
	preserves finite limits and all small colimits by \cref{prop: derivatives-lmod-colim}, it preserves the commutator sequence.
    This is a cofiber sequence in pointed spaces by \cref{lem: spaces-comm-cofiber-seq}, so we learn that it is also one in spectral Lie algebras.
    The fact that it is of the required form for free spectral Lie algebras follows at once by combining \cref{lem: derivatives-smash-product,lem: lie-bracket-is-whitehead}.
\end{proof}

\begin{corollary}[\cref{cor: intro-smash-associative}]
    The cartesian product 
    \[
    \times \colon \Lie(\Sp) \times \Lie(\Sp) \to \Lie(\Sp)
    \]
    commutes with weakly contractible colimits in both variables.
    Consequently, the smash product defines a nonunital symmetric monoidal structure on $\Lie(\Sp)$.
\end{corollary}
\begin{proof}
    The commutator cofiber sequence for free spectral Lie algebras implies that the functor
    \[
    (A, B) \mapsto \free_\LL(A) \times \free_\LL(B)
    \]
    preserves weakly contractible colimits in both variables.
    Writing an arbitrary Lie algebra as a sifted colimit of free Lie algebras, we find that the cartesian product commutes with weakly contractible colimits in both variables.
    This immediately implies that the smash product defines a nonunital symmetric monoidal structure by \cref{prop: smash-monoidal-structure}.
\end{proof}

\begin{remark}
Since $\free_\LL$ takes the tensor product of spectra to the smash product of Lie algebras, the reader might be led to believe that $\free_\LL(\Sph)$ is a unit for the smash product. This turns out not to be the case: one can show that for $X \in \Alg_\LL(\Sp)$ there is an isomorphism
\[
X \wedge \free_\LL(\Sph) \simeq \free_\LL(\cot_\LL(X)),
\]
so that smashing with $\free_\LL(\Sph)$ only acts as a unit on free algebras. We will not need this fact here, so we do not pursue it further.
\end{remark}

\subsection{The James construction}

The aim of this section is to prove \Cref{introthm:James}. We will begin by showing that any monoid object in spectral Lie algebras is automatically group-complete:

\begin{proposition}
\label{prop:freeE1AlgO}
Let $\sC$ be a stable presentably symmetric monoidal $\infty$-category and $\sO$ a nonunital operad in $\sC$. Then the map of $\mathbb{E}_1$-monoids $\free_{\E_1}(X) \to \Omega\Sigma X$, induced by the unit $X \to \Omega\Sigma X$, is an equivalence. Moreover, the two functors
\[
\Alg_{\sO}(\sC) \xrightarrow{\Omega} \mathrm{Grp}(\Alg_{\sO}(\sC)) \xrightarrow{\mathrm{forget}} \mathrm{Mon}(\Alg_{\sO}(\sC))
\]
are equivalences of $\infty$-categories, where the last term denotes the $\infty$-category of monoid objects in $\Alg_{\sO}(\sC)$ and the middle term the analogous $\infty$-category of group objects. 
\end{proposition}
\begin{remark}
In particular, any monoid of $\Alg_{\sO}(\sC)$ is automatically grouplike, and any group object $X \in \mathrm{Grp}(\Alg_{\sO}(\sC))$ admits an essentially unique delooping. Observe that in the $\infty$-category of pointed spaces such statements need the hypothesis that $X$ is connected, but they hold unconditionally in the setting of $\sO$-algebras. 
\end{remark}
\begin{proof}[Proof of \Cref{prop:freeE1AlgO}]
Consider the commutative diagram of right adjoint functors
\[
\begin{tikzcd}
\Alg_{\sO}(\sC)\ar{r}{\Omega}\ar{d}{\mathrm{forget}} & \mathrm{Mon}(\Alg_{\sO}(\sC)) \ar{r}{\mathrm{forget}}\ar{d}{\mathrm{forget}} & \Alg_{\sO}(\sC) \ar{d}{\mathrm{forget}} \\
\sC \ar{r}{\Omega} & \sC \ar[equal]{r} & \sC.
\end{tikzcd}
\]
The left adjoint of the top-left functor $\Omega$ is the usual bar construction of a monoid:
\[
\mathrm{Bar}\colon \mathrm{Mon}(\Alg_{\sO}(\sC)) \to \Alg_{\sO}(\sC)\colon X \mapsto \mathrm{colim}_{\mathbf{\Delta}^{\mathrm{op}}} X^{\bullet}.
\]
Since the forgetful functor $\Alg_{\sO}(\sC) \to \sC$ preserves products and sifted colimits, we find $\mathrm{forget} \circ \mathrm{Bar} \simeq \Sigma \circ \mathrm{forget}$ (since $\sC$ is stable, so bar construction and suspension agree). Hence the unit and counit of the adjoint pair $(\mathrm{Bar},\Omega)$ become, after applying the forgetful functor, the unit and counit of the loops-suspension adjunction on $\sC$ and hence equivalences. It follows that the top-left functor $\Omega$ is indeed an equivalence of $\infty$-categories. Since $\Omega$ takes values in grouplike monoids, it follows that every monoid object is grouplike. This proves that the two functors in the statement of the proposition are indeed equivalences. Finally, the right adjoint functor 
\[
\mathrm{Mon}(\Alg_{\sO}(\sC)) \xrightarrow{\mathrm{forget}} \Alg_{\sO}(\sC)
\]
induces a monad on $\Alg_{\sO}(\sC)$, namely the free $\E_1$-monoid functor $\free_{\E_1}$. Similarly, the right adjoint functor 
\[
\mathrm{Mon}(\Alg_{\sO}(\sC)) \xrightarrow{\Omega} \Alg_{\sO}(\sC)
\]
induces the monad $\Omega\Sigma$. The diagram at the start of this proof and our earlier argument show that the comparison map $\free_{\E_1} \to \Omega\Sigma$ between the two is indeed an equivalence.
\end{proof}

It is proved in \cite{blanslinskens} that the James construction $\mathcal{J}(X)$ computes the free $\E_1$-monoid on $X$, provided that the cartesian product commutes with weakly contractible colimits in each variable. This is the case in $\Lie(\Sp)$ by \Cref{cor: intro-smash-associative}, so that the first part of  \Cref{introthm:James} follows immediately. However, let us also give a direct proof using the methods of this paper:

\begin{proof}[Proof of \Cref{introthm:James}]
First consider the usual natural transformation $\mathcal{J} \to \Omega\Sigma$ between functors $\Spc_* \to \Spc_*$. It is a classical fact that this map is an equivalence when evaluated on a connected space $X$, and that there is a natural James splitting 
$\Sigma\mathcal{J}(X) \simeq \bigvee_{n \geq 1} \Sigma X^{\wedge n}$
in that case. We now apply the Goodwillie transform $\maAlg$ to obtain corresponding natural transformations between functors $\Lie(\Sp) \to \Lie(\Sp)$. Since $\maAlg$ preserves colimits and finite limits on mapping categories (\Cref{cor:exactnessGoodwillietransform}), it preserves the James construction and sends $\Omega\Sigma$ on $\Spc_*$ to the corresponding construction on $\Lie(\Sp)$. Moreover, if $F \to G$ is a natural transformation of functors $\Spc_* \to \Spc_*$ that is an equivalence on connected spaces, then it induces an equivalence $\partial_* F \simeq \partial_* G$ on derivatives. Hence the same is true for $\maAlg(F) \to \maAlg(G)$. This completes the proof.
\end{proof}

\subsection{The EHP sequence}
\label{subsec:EHP}

Let $X$ be a connected pointed space and consider the usual EHP sequence, which consists of natural maps
\[
X \xrightarrow{E} \Omega\Sigma X \xrightarrow{H} \Omega\Sigma X^{\wedge 2}
\]
and a natural nullhomotopy of the composite. Here $E$ is the unit of the loops-suspension adjunction and the Hopf map $H$ is adjoint to the map
\[
\Sigma\Omega\Sigma X \simeq \Sigma \bigvee_n X^{\wedge n} \to \Sigma X^{\wedge 2}
\]
combining the James splitting and projection onto the second factor. We now apply the Goodwillie transform $\maAlg$ to obtain a corresponding null sequence of functors
\[
\mathrm{id}_{\Lie(\Sp)} \xrightarrow{E} \Omega\Sigma \xrightarrow{H} \Omega\Sigma(-)^{\wedge 2}
\]
on the $\infty$-category $\Lie(\Sp)$. Here we have again used that $\maAlg$ preserves colimits and finite limits of functors  (\Cref{cor:exactnessGoodwillietransform}) to conclude that the sequence takes the same shape on $\Lie(\Sp)$. In fact, one can argue that the maps $E$ and $H$ admit the same description as before, but now interpreted inside $\Lie(\Sp)$, relying on the James splitting for spectral Lie algebras that we produced above.

\begin{proof}[Proof of \Cref{introthm:EHP}]
It remains to show that the sequence is a 2-local fiber sequence when evaluated on a spectral Lie algebra $X = \free_{\LL}(\mathbb{S}^k)$. This can be checked on derivatives: we should argue that
\[
\partial_*(\mathrm{id}_{\Spc_*}) \circ \mathbb{S}^k \xrightarrow{E} \partial_*(\Omega\Sigma) \circ \mathbb{S}^k  \xrightarrow{H} \partial_*(\Omega\Sigma(-)^{\wedge 2}) \circ \mathbb{S}^k 
\]
is a 2-local fiber sequence. This has been shown by Behrens \cite[Lemma 2.1.2]{behrensEHP} and by Arone--Mahowald \cite[Propositions 4.6, 4.7]{aronemahowald}.
\end{proof}

\subsection{The Hilton--Milnor splitting}

For pointed connected spaces $X$ and $Y$, the classical Hilton--Milnor theorem provides a splitting of the loop space $\Omega\Sigma(X \vee Y)$ as an infinite product, indexed by a basis for the free Lie algebra on 2 generators, of factors of the form $\Omega\Sigma(X^{\wedge a} \wedge Y^{\wedge b})$. We will now derive the analogous splitting in the $\infty$-category of spectral Lie algebras. At the level of derivatives this was already observed by Arone--Kankaanrinta \cite{aronekankaanrinta} and a statement on Goodwillie towers was considered by Brantner--Heuts \cite{brantnerheuts}.

In more detail, write $L_2$ for an ordered set of Lie words in two letters $x, y$ forming a basis for the free classical Lie algebra (say over $\mathbb{Q}$) generated by $x$ and $y$. For $X, Y \in \Lie(\Sp)$ and a word $w \in L_2$, we define $w(X,Y)$ by letting the bracket act as the smash: e.g., if $w = {[x,[x,y]]}$ then $w(X,Y) = X^{\wedge 2} \wedge Y$. For each $w$ we define a map
\[
w(X,Y) \to \Omega\Sigma(X \vee Y)
\]
using the (iterated) Whitehead bracket defined by $w$. This extends to a map of group objects 
\[
\Omega\Sigma w(X,Y) \to \Omega\Sigma(X \vee Y).
\]
Now multiplying these maps together using the group structure of the codomain, we obtain a map
\[
\varphi\colon \sideset{}{'}\prod_{w \in L_2} \Omega\Sigma w(X,Y) \to \Omega\Sigma(X \vee Y)
\]
where $\sideset{}{'}\prod$ denotes the weak infinite product, meaning the filtered colimit of products indexed by finite subsets of $L_2$.

\begin{theorem}
\label{thm:HiltonMilnor}
The map $\varphi$ is an equivalence of spectral Lie algebras.
\end{theorem}
\begin{proof}
The Hilton--Milnor theorem states that the analogous construction in the $\infty$-category of pointed spaces yields an equivalence when evaluated on connected $X$ and $Y$. The conclusion follows immediately by applying the Goodwillie transform $\maAlg$, using that it preserves colimits and finite limits of functors and the fact that an equivalence on connected spaces yields an equivalence on derivatives, all as before (e.g. in the proof of \Cref{introthm:James}).
\end{proof}

\subsection{Universality of weakly contractible colimits} 
Another typical feature of the homotopy theory of unpointed spaces is the universality of certain kinds of colimits.

\begin{definition}
    Let $\mathcal{I}$ be a small category and let $\sC$ be a category that admits pullbacks and $\mathcal{I}$-indexed colimits.
    We say that $\mathcal{I}$-indexed colimits are universal in $\sC$ if for every morphism $f \colon X \to Y$ in $\sC$, the pullback functor
    \[
    f^* \colon \sC_{/Y} \to \sC_{/X}
    \]
    preserves $\mathcal{I}$-indexed colimits.
    In other words, for any diagram $g \colon \mathcal{I} \to \sC_{/Y}$, the comparison map
    \[
    \mathrm{colim}_{\mathcal{I}} (X \times_Y g(-)) \to X \times_Y \mathrm{colim}_{\mathcal{I}} g(-)
    \]
    is an isomorphism.
\end{definition}

\begin{example}
    In the category $\Spc$ of spaces, $\mathcal{I}$-indexed colimits are universal for any small indexing category $\mathcal{I}$.
    Since the forgetful functor $\Spc_* \to \Spc$ preserves and reflects all limits and weakly contractible colimits, $\mathcal{I}$-indexed colimits are universal in $\Spc_*$ whenever $\mathcal{I}$ is weakly contractible.
\end{example}

The main result of this section is that the category of spectral Lie algebras behaves the same as the category of pointed spaces when it comes to universality of colimits.

\begin{theorem}[\cref{introthm:Mather}]\label{thm: universality-of-colimits-Lie}
    Let $\mathcal{I}$ be a weakly contractible small category.
    Then $\mathcal{I}$-indexed colimits are universal in $\Lie(\Sp)$.
\end{theorem}

\begin{corollary}[Mather's second cube lemma]
Suppose that we have a commutative cube of spectral Lie algebras
\[
    \begin{tikzcd}[row sep=scriptsize, column sep=scriptsize]
& A \arrow[dl] \arrow[rr] \arrow[dd] & & C \arrow[dl] \arrow[dd] \\
B \arrow[rr, crossing over] \arrow[dd] & & D \\
& W \arrow[dl] \arrow[rr] & & Y \arrow[dl] \\
X \arrow[rr] & & Z \arrow[from=uu, crossing over]\\
\end{tikzcd}
\]
where the bottom face is cocartesian and all the vertical faces are cartesian.
Then the top face is cocartesian as well.
\end{corollary}

Our proof of \cref{thm: universality-of-colimits-Lie} will of course amount to differentiating the universality of weakly contractible colimits in pointed spaces.
In order to carry out this proof, we need two additional properties of the functor $\maAlg \colon \diff \to \diff$, both of which will be proved in \cite{BlansBlom2026}.
The first one is:

\begin{proposition} \label{prop: malin-preserves-cotensors}
    The functor $\maAlg \colon \diff \to \diff$ preserves cotensors by small $\infty$-categories. 
\end{proposition}

\begin{remark}
Observe that the $(\infty,2)$-category $\diff$ admits such cotensors: if $\sC$ is a differentiable $\infty$-category and $K$ is a small $\infty$-category, then $\Fun(K, \sC)$ is again differentiable and provides a cotensor of $\sC$ by $K$ in $\diff$. 
\end{remark}

The proposition implies that
\[
\maAlg(\Fun(K, \sC)) \simeq \Fun(K, \Alg_{\partial_*{\id_\sC}}),
\]
where we have abbreviated $\Alg_{\partial_*{\id_\sC}}(\Sp(\sC))$ to $\Alg_{\partial_*{\id_\sC}}$.
It also implies that if $p \colon I \to J$ is a functor of small $\infty$-categories, then the restriction functor $p^* \colon \Fun(J, \sC) \to \Fun(I, \sC)$ is preserved by $\maAlg$.
Since any functor of $(\infty,2)$-categories preserves adjunctions, this has the further consequence that $\maAlg$ sends the colimit functor $\Fun(K, \sC) \to \sC$ to
\[
\colim \colon \Fun(K, \Alg_{\partial_*{\id_\sC}}) \to \Alg_{\partial_*{\id_\sC}}.
\]
The same holds for the limit functor if $K$ is finite (if $K$ is not finite, the limit functor is not finitary).

The other property of $\maAlg$ we need is:

\begin{proposition} \label{prop: malin-preserves-certain-pullbacks}
    Let
    \[
    \begin{tikzcd}
        \sA \ar[r, "\bar{p}^*", hook] \ar[d, "\bar{q}^*"']  \ar[dr, phantom, "\lrcorner", very near start] & \sC \ar[d, "q^*"] \\
        \sB \ar[r, "p^*"', hook] & \sD
    \end{tikzcd}
    \]
    be a pullback square of $\infty$-categories where $\sB$, $\sC$, and $\sD$ are differentiable and $p^*$ and $q^*$ are reduced and finitary.
    Assume that
    \begin{enumerate}[\upshape{(}\arabic*\upshape{)}]
        \item $p^*$ and $q^*$ both admit a left adjoint, and
        \item $p^*$ is fully faithful.
    \end{enumerate}
    Then $\sA$ is differentiable, $\bar{p}^*$ and $\bar{q}^*$ are reduced and finitary, and $\maAlg \colon \diff \to \diff$ preserves the pullback square.
\end{proposition}

With this in place, we come to the proof of universality of weakly contractible colimits in spectral Lie algebras.

\begin{proof}[Proof of \cref{thm: universality-of-colimits-Lie}]
    Let $\mathcal{I}$ be a weakly contractible small $\infty$-category as in the statement of the theorem.
    We start by rephrasing the universality of $\mathcal{I}$-indexed colimits in $\Spc_*$ in a way that is amenable to applying Goodwillie calculus.
    Let $K = (a \rightarrow b \leftarrow c)$ be the free cospan and consider the pullback of categories
    \begin{equation}\label{eq: universality-pullback-square}
    \begin{tikzcd}
        \sE \ar[d] \ar[r, hook] \ar[dr, phantom, "\lrcorner", very near start] & \Fun(K \times \mathcal{I}, \Spc_*) \ar[d, "q^*"] \\
        \Fun(\{a \to b\}, \Spc_*) \ar[r, "p^*"', hook] & \Fun(\{a \to b\} \times \mathcal{I}, \Spc_*),
    \end{tikzcd}
    \end{equation}
    where the right hand map is restriction along the inclusion $q \colon \{a \to b\} \times \mathcal{I} \to K \times \mathcal{I}$, and the bottom map is restriction along the projection $p \colon \{a \to b\}\times \mathcal{I} \to \{a \to b\}$.
    Note that $p^*$ is fully faithful since $\mathcal{I}$ is weakly contractible.
    The $\infty$-category $\sE$ consists of cospans of $\mathcal{I}$-indexed diagrams of the form
    \[
    p^*(A) \xrightarrow{p^*(f)} p^*(B) \xleftarrow{\phantom{p}g\phantom{p}} C,
    \]
    i.e.\ the left arrow is a constant natural transformation between constant $\mathcal{I}$-indexed diagrams.
    Now consider the canonical natural transformation
    \begin{equation} \label{eq: limit-colimit-interchange-square}
    \begin{tikzcd}
        \Fun(K \times \mathcal{I}, \Spc_*) \ar[r, "\lim_K"] \ar[d, "\colim_{\mathcal{I}}"'] & \Fun(\mathcal{I}, \Spc_*) \ar[d, "\colim_{\mathcal{I}}"] \ar[dl, "\alpha", Rightarrow, shorten=1em] \\
        \Fun(K, \Spc_*) \ar[r, "\lim_K"'] & \Spc_*
    \end{tikzcd}
    \end{equation}
    that interchanges the limit and the colimit.
    Observe that the universality of $\mathcal{I}$-indexed colimits in $\Spc_*$ is equivalent to $\alpha$ becoming an equivalence after restricting to the full subcategory $\sE \subseteq \Fun(K \times \mathcal{I}, \Spc_*)$.

    We now apply $\maAlg$ to the squares \cref{eq: universality-pullback-square} and \cref{eq: limit-colimit-interchange-square}.
    For $\cref{eq: universality-pullback-square}$, observe that all three $\infty$-categories in the cospan are differentiable, $p^*$ and $q^*$ are reduced finitary right adjoints, and $p^*$ is fully faithful.
    It follows from \cref{prop: malin-preserves-certain-pullbacks} that $\maAlg$ sends this square to a pullback square.
    By \cref{prop: malin-preserves-cotensors}, we can identify the resulting square with
    \[
    \begin{tikzcd}
    \Alg_{\partial_*{\id_\sE}} \ar[d] \ar[r, hook] \ar[dr, phantom, "\lrcorner", very near start] & \Fun(K \times \mathcal{I}, \Lie(\Sp)) \ar[d, "q^*"] \\
        \Fun(\{a \to b\}, \Lie(\Sp)) \ar[r, "p^*"', hook] & \Fun(\{a \to b\} \times \mathcal{I}, \Lie(\Sp)).
    \end{tikzcd}
    \]
    Another application of \cref{prop: malin-preserves-cotensors} shows that $\maAlg$ sends the square \cref{eq: limit-colimit-interchange-square} to the same square with $\Spc_*$ replaced by $\Lie(\Sp)$.
    The natural transformation $\maAlg(\alpha)$ is again the canonical map interchanging the colimit and the limit; this follows because functors of $(\infty,2)$-categories preserve Beck-Chevalley transformations.
    By functoriality of $\maAlg$, we find that this transformation becomes an equivalence upon restricting to the full subcategory
    \[
    \Alg_{\partial_*{\id_\sE}} \subseteq \Fun(K \times \mathcal{I}, \Lie(\Sp)).
    \]
    But this is equivalent to the universality of $\mathcal{I}$-indexed colimits in spectral Lie algebras.
\end{proof}

\section{The characterization}

In this section we will prove the characterization of the spectral Lie operad.
In \cref{sec: lax-sym-mon-comonads}, we give a short preliminary discussion of lax symmetric monoidal comonads.
We then give a characterization of the cocommutative cooperad in \cref{sec: characterization-cocom-cooperad}.
Using this, we show that the derivatives of the identity functor in pointed spaces are the spectral Lie operad in \cref{sec: derivatives-identity-spaces}.
We then prove in \cref{sec: free-Lie-algebra} that the free spectral Lie algebra functor is symmetric monoidal with respect to the tensor product of spectra and the smash product of Lie algebras.
Finally, we prove the characterization in \cref{sec: proof-characterization}.

\subsection{Lax symmetric monoidal comonads} \label{sec: lax-sym-mon-comonads}
Our proof of the characterization in the end relies on the fact that the cartesian product in the $\infty$-category $\coCAlg(\Sp)$ of cocommutative coalgebras is given by the tensor product of spectra.
It will therefore be useful to know when one can lift the tensor product to a symmetric monoidal structure on the $\infty$-category of $\sQ$-coalgebras, where $\sQ$ is a cooperad. We address this question now.

Let $\sC$ be a symmetric monoidal $\infty$-category and suppose that $T \colon \sC \to \sC$ is a comonad. 
If the functor $T$ has a lax symmetric monoidal structure, then we could try to define a $T$-coalgebra structure on the tensor product of two $T$-coalgebras $A$ and $B$ by taking the composite
\[
A \otimes B \to T(A) \otimes T(B) \to T(A \otimes B)
\]
as the coaction of $T$ on $A \otimes B$. 
Here, the first morphism is the tensor product of the coactions of $T$ on $A$ and $B$ respectively, and the second morphism uses that $T$ is a lax monoidal functor.
Of course, for this to work we need the lax monoidal structure to be compatible with the comonad structure.
The precise definition is as follows.

\begin{definition}
    Let $\sC$ be a symmetric monoidal $\infty$-category and write $\End^{\mathrm{lax}}(\sC)$ for the $\infty$-category of lax symmetric monoidal endofunctors on $\sC$.
    This is a monoidal $\infty$-category under composition.
    A \emph{lax symmetric monoidal comonad} is a coalgebra in $\End^{\mathrm{lax}}(\sC)$.
\end{definition}

This definition indeed accomplishes what it was designed for, as we have:

\begin{proposition}\label{prop: oplax-monad-tensor-algebras}
    Let $T$ be a lax symmetric monoidal comonad on $\sC$. 
    Then $\coAlg_{T}(\sC)$ admits a symmetric monoidal structure for which the forgetful functor $\coAlg_T(\sC) \to \sC$ is strong monoidal.
\end{proposition}
\begin{proof}
    This is proved in \cite[Theorem 1.11]{heine2025monadicitytheoremhigheralgebraic}.
\end{proof}

The main source of examples of lax symmetric monoidal comonads is the following:

\begin{proposition} \label{prop: monoidal-adjunction-comonad}
    Let
    \[
    \begin{tikzcd}
        F \colon \sC \ar[r, shift left] & \sD \ar[l, shift left] : G
    \end{tikzcd}
    \]
    be an adjunction between symmetric monoidal $\infty$-categories where $F$ is strong symmetric monoidal.
    Then the comonad $FG$ lifts to a lax symmetric monoidal comonad.
\end{proposition}
\begin{proof}
    Since $F$ is strong symmetric monoidal, we obtain a lax symmetric monoidal structure on $G$ such that the unit and counit refine to lax symmetric monoidal natural transformations \cite[Corollary 7.3.2.7]{HA}.
    One then observes that the unit $F \to FGF$ exhibits $FG \in \End^{\mathrm{lax}}(\sD)$ as the coendomorphism object of $F \in \Fun^{\mathrm{lax}}(\sC, \sD)$; the proof is exactly the same as in \cite[Lemma 4.7.3.1]{HA}.
    Since a coendomorphism object acquires a canonical coalgebra structure by \cite[Corollary 4.7.1.41]{HA}, this completes the proof.
\end{proof}

When $\sC$ is a cartesian symmetric monoidal $\infty$-category and $\sD$ is $\Sp$ with the augmented tensor product, it turns out that the tensor product $\otimesaug$ on the $\infty$-category of $FG$-coalgebras is equivalent to the cartesian product.
More precisely, we have the following:

\begin{proposition} \label{prop: lax-sm-comonad-over-sp-aug}
    Let $\sD$ be a stable presentably symmetric monoidal $\infty$-category and let $T$ be a comonad on $\sD$. The following are equivalent:
    \begin{enumerate}[\upshape{(}\arabic*\upshape{)}]
        \item There is a pointed $\infty$-category $\sC$ with finite limits and colimits and an adjunction $F \colon \sC \rightleftharpoons \sD : G$, where $F$ is a nonunital symmetric monoidal functor with respect to the smash product on $\sC$ and the tensor product $\otimes$ on $\sD$, and there is an equivalence of comonads $T \simeq FG$.
        \item There is a pointed $\infty$-category $\sC$ with finite limits and colimits and an adjunction $F \colon \sC \rightleftharpoons \sD : G$, where $F$ is a symmetric monoidal functor with respect to the cartesian product on $\sC$ and the augmented tensor product $\otimesaug$ on $\sD$, and there is an equivalence of comonads $T \simeq FG$.
        \item The comonad $T$ lifts to a lax symmetric monoidal comonad for $\otimesaug$, and for every pair $X, Y \in \sD$, the composite
        \[
        T(X) \otimesaug T(Y) \to T(X \otimesaug Y) \to T(X \times Y)
        \]
        is an isomorphism, where the first arrow is the lax comparison map of $T$ and the second one is induced by the projection $X \otimesaug Y \to X \times Y$.
        \item The forgetful functor $\coAlg_T(\sD) \to \sD$ admits a symmetric monoidal structure from $\times$ to $\otimesaug$.
        \item The forgetful functor $\coAlg_T(\sD) \to \sD$ admits a nonunital symmetric monoidal structure from the smash product $\wedge$ to $\otimes$.
    \end{enumerate}
\end{proposition}
\begin{proof}
    $(1) \Rightarrow (2)$: The composite of functors
    \begin{align*}
    \sC & \xrightarrow{\phantom{uFu}} \coCAlg^{\mathrm{nu}}(\sC, \times) \\
    & \xrightarrow{\phantom{uFu}} \coCAlg^{\mathrm{nu}}(\sC, \wedge) \\
    &\xrightarrow{\phantom{u}F\phantom{u}} \coCAlg^{\mathrm{nu}}(\sD, \otimes).
    \end{align*}
    provides a lift of $F \colon \sC \to \sD$ through the category $\coCAlg^{\mathrm{nu}}(\sD)$.
    Here, the first arrow uses that every object in a category with products is canonically a nonunital commutative coalgebra for the cartesian product \cite[Corollary 2.4.3.10]{HA}, the second arrow is given by the oplax monoidal structure from $\times$ to $\wedge$ on the identity functor of $\sC$, and the third arrow uses the assumption that $F$ is symmetric monoidal from $\wedge$ to $\otimes$.

    The resulting functor $\sC \to \coCAlg^{\mathrm{nu}}(\sD)$ is canonically oplax monoidal for the cartesian symmetric monoidal structure on source and target, since this holds for any functor between $\infty$-categories with finite products \cite[Proposition 2.4.3.16]{HA}.
    By composing with the forgetful functor $\coCAlg^{\mathrm{nu}}(\sD) \to \sD$, which is strong monoidal from $\times$ to $\otimesaug$, we obtain an oplax monoidal structure on the functor $F \colon \sC \to \sD$ with respect to $\times$ and $\otimesaug$.
    For every pair of objects $M, N \in \sC$, the oplax comparison map by construction fits in a commutative diagram
    \[
    \begin{tikzcd}[sep = small]
        F(M \sqcup N) \ar[r] \ar[d, "\wr"] & F(M \times N) \ar[r] \ar[d] & F(M \wedge N) \ar[d, "\wr"] \\[1em]
        F(M) \oplus F(N) \ar[r] & F(M) \otimesaug F(N) \ar[r] & F(M) \otimes F(N).
    \end{tikzcd}
    \]
    The left vertical arrow is an isomorphism since $F$ preserves colimits and the right vertical arrow is an isomorphism by assumption, so that the middle vertical arrow is an isomorphism by the $5$-lemma.
    This shows that $F$ admits a symmetric monoidal structure from $\times$ to $\otimesaug$.
    
    $(2) \Rightarrow (3)$: It follows from \cref{prop: monoidal-adjunction-comonad} that $T$ acquires the structure of a lax symmetric monoidal comonad for $\otimesaug$.
    The symmetric monoidal structure on $F$ induces a lax symmetric monoidal structure on $G$ with structure map given by the composite
    \begin{align*}
    G(X) \times G(Y) &\xrightarrow{\phantom{u}\eta\phantom{u}} GF(G(X) \times G(Y)) \\
    & \xrightarrow{\phantom{r}\sim \phantom{r}} G(FG(X) \otimesaug FG(Y)) \\
    & \xrightarrow{\phantom{ }\epsilon \otimes \epsilon\phantom{ }} G(X \otimesaug Y),
    \end{align*}
    where $\eta$ and $\epsilon$ denote respectively the unit and counit of the adjunction $(F, G)$.
    It is an immediate consequence of the triangle identities and the fact that $G$ preserves products that the composite of this map with the projection $G(X \otimesaug Y) \to G(X \times Y)$ is an isomorphism.
    Applying $F$ to this total composite and using that $FG \simeq T$, we find that the composite
    \[
        T(X) \otimesaug T(Y) \to T(X \otimesaug Y) \to T(X \times Y),
    \]
    is an isomorphism.

    $(3) \Rightarrow (4)$: Since $T$ is a lax symmetric monoidal comonad, the augmented tensor product lifts to a symmetric monoidal structure on $\coAlg_T(\sD)$. 
    For any pair $M, N \in \coAlg_{T}(\sD)$ there is a natural comparison morphism $M \otimesaug N \to M \times N$ with components 
    \[
    M \simeq M \otimesaug 0 \leftarrow M \otimesaug N \rightarrow 0 \otimesaug N  \simeq N.
    \]
    Since $\forget \colon \coAlg_T(\sD) \to \sD$ is symmetric monoidal for $\otimesaug$, it suffices to show that this comparison map is an isomorphism.
    Using the standard cobar resolution, we write $M \simeq \Tot C^{\bullet}(T, T, M)$ as a totalization of cofree $T$-coalgebras.
    Consider the following commutative diagram of coalgebras:
    \[
    \begin{tikzcd}
    	\Tot (C^{\bullet}(T, T, M) \otimesaug N) \ar[r] \ar[d] & M \otimesaug N \ar[d] \\
    	\Tot(C^{\bullet}(T, T, M) \times N) \ar[r] & M \times N.
    \end{tikzcd}
    \]
    The top horizontal arrow is an isomorphism since the underlying cosimplicial object of $C^{\bullet}(T, T, M) \otimesaug N$ is split.
    The bottom horizontal arrow is an isomorphism since totalizations commute with products.
    It therefore suffices to prove the map is an isomorphism if $M$ is a cofree $T$-coalgebra.
    By the same argument, we may also assume $N$ is cofree. 
    But in this case, the map is equivalent to the composite
    \[
    \cofree_{T}(X) \otimesaug \cofree_{T}(Y) \to \cofree_{T}(X \otimesaug Y) \to \cofree_{T}(X \times Y),
    \]
    which we assumed to be an isomorphism.

    $(4) \Rightarrow (5)$: This follows from \cref{rem: right-exact-strong-monoidal-smash}, since the forgetful functor preserves colimits.
    
    $(5) \Rightarrow (1)$: The cofree functor is right adjoint to the forgetful functor.
\end{proof}

\begin{corollary}\label{cor: exponential-property}
    Let $\sD$ be a stable presentably symmetric monoidal category and let $T$ be a comonad on $\sD$ satisfying the equivalent conditions from the previous proposition.
    Then $T$ admits a symmetric monoidal structure from $\oplus$ to $\otimesaug$.
\end{corollary}

\begin{example} \label{ex: sym-cocom-lax-sm-comonad}
     The forgetful functor $\coCAlg(\SSeqnu(\Sp)) \to \SSeqnu(\Sp)$ admits a symmetric monoidal structure from $\times$ to $\otimesaug$ by \cref{prop: product-nu-cocalg}.
    In this case, the previous corollary recovers the well-known fact that the functor $\Sym_{\coCom}$ is symmetric monoidal from $\oplus$ to $\otimes^{\mathrm{aug}}$.
\end{example}

\subsection{A characterization of the cocommutative cooperad} \label{sec: characterization-cocom-cooperad}

The cocommutative cooperad admits the following characterization.

\begin{proposition} \label{prop: characterization-cocom}
	Let $\sQ$ be a reduced cooperad in spectra.
	Then the following are equivalent:
	\begin{enumerate}[\upshape{(}\arabic*\upshape{)}]
    \item The cooperad $\sQ$ is equivalent to the cocommutative cooperad $\mathbf{Com}^{\vee}$.
    \item The functor
    \[
    \forget_{\sQ} \colon \LcoMod_{\sQ} \to \SSeqnu(\Sp)
    \]
    admits a nonunital symmetric monoidal structure with respect to the smash product on $\LcoMod_{\sQ}$ and the Day convolution tensor product on $\SSeqnu(\Sp)$.
\end{enumerate}
\end{proposition}

\begin{proof}
	The implication (1) $\Rightarrow$ (2) follows immediately from \cref{prop: product-nu-cocalg,rem: right-exact-strong-monoidal-smash}.

    We now prove the other implication.
    First of all, since the forgetful functor $\LcoMod_{\sQ} \to \SSeqnu(\Sp)$ admits a nonunital symmetric monoidal structure from $\wedge$ to $\otimes$, it also admits a symmetric monoidal structure from the cartesian product to $\otimesaug$ by \cref{prop: lax-sm-comonad-over-sp-aug}.
    
    Since every object in $\LcoMod_{\sQ}$ is canonically a commutative coalgebra for the cartesian product and the forgetful functor $\LcoMod_{\sQ} \to \SSeqnu(\Sp)$ is symmetric monoidal from $\times$ to $\otimesaug$, we obtain a factorization of the forgetful functor through $\coCAlg^{\mathrm{nu}}(\SSeqnu(\Sp))$, so that we get a commutative diagram
    \[
    \begin{tikzcd}
        \LcoMod_{\sQ} \ar[rr, "H"] \ar[dr, "\forget_{\sQ}"'] & & \coCAlg^{\mathrm{nu}}(\SSeqnu(\Sp)) \ar[dl, "\forget_{\coCom}"] \\
        & \SSeqnu(\Sp). &
    \end{tikzcd}
    \]
    Observe that all functors in this diagram are symmetric monoidal for $\otimesaug$, and that the diagram commutes as a diagram of symmetric monoidal functors.
    The diagram gives rise to a map
    \[
    \Sym_{\sQ} \to \Sym_{\coCom}
    \]
    of lax symmetric monoidal comonads for $\otimesaug$.
    Indeed, since the (lax) symmetric monoidal functor $\forget_{\coCom}$ has the structure of a left comodule over the lax symmetric monoidal comonad $\Sym_{\coCom}$, the (lax) symmetric monoidal functor $\forget_\sQ \simeq \forget_\Com \circ H$ also acquires such a left module structure, and this is classified by a map $\Sym_{\sQ} \to \Sym_{\coCom}$ of lax symmetric monoidal comonads, since $\Sym_{\sQ}$ is the coendomorphism object of $\forget_{\sQ} \in \Fun^{\mathrm{lax}}(\LcoMod_{\sQ}, \SSeqnu(\Sp))$.
    Moreover, by \cref{cor: exponential-property} and \cref{ex: sym-cocom-lax-sm-comonad}, both $\Sym_{\coCom}$ and $\Sym_{\sQ}$ acquire a symmetric monoidal structure from $\oplus$ to $\otimesaug$ and the map $\Sym_{\sQ} \to \Sym_{\coCom}$ therefore refines to a natural transformation of symmetric monoidal functors $(\SSeqnu(\Sp), \oplus) \to (\SSeqnu(\Sp), \otimesaug)$.
    In particular, for every pair of symmetric sequences $X, Y \in \SSeqnu(\Sp)$, we get a commutative square
    \[
    \begin{tikzcd}
         \Sym_{\sQ}(X \oplus Y) \ar[r] \ar[d, "\wr"] & \Sym_{\coCom}(X \oplus Y) \ar[d, "\wr"] \\
        \Sym_{\sQ}(X) \otimesaug \Sym_{\sQ}(Y)  \ar[r] & \Sym_{\coCom}(X) \otimesaug \Sym_{\coCom}(Y).
    \end{tikzcd}
    \]
    
    The map of comonads $\Sym_{\sQ} \to \Sym_{\coCom}$ yields a map of cooperads $B\sO \to \coCom$ by passage to Goodwillie derivatives by \cref{ex: der-retraction-sym}.
    We will now show by induction that $\sQ(n) \to \coCom(n)$ is an isomorphism for all $n \geq 1$.
    For $n=1$, there is nothing to show since the cooperads in question are reduced.
    If $A$ is a symmetric sequence and $X, Y \in \SSeqnu(\Sp)$, we have natural decompositions
    \begin{align*}
    \Sym_A(X \oplus Y) &\simeq \bigoplus_{k, \ell \geq 0} \mathrm{Res}^{\Sigma_{k+\ell}}_{\Sigma_k \times \Sigma_\ell} A(k + \ell) \otimes_{h\Sigma_k \times \Sigma_\ell} X^{\otimes k} \otimes Y^{\otimes \ell}, \\
    \Sym_A(X) \otimesaug \Sym_A(Y) &\simeq \bigoplus_{k, \ell \geq 0} A(k) \otimes A(\ell) \otimes_{h \Sigma_k \times \Sigma_\ell} X^{\otimes k} \otimes Y^{\otimes \ell}.
    \end{align*}
    In the final expression, one should set $A(0) = 0$ if $k = \ell = 0$ and $A(0) = \Sph$ otherwise.
    Note that in both sums, the $(k,\ell)$-component is a $(k,\ell)$-homogeneous functor.
    Applying the multivariable Goodwillie derivative functor $\partial_{(n-1, 1)}$ to the commutative square above, we therefore obtain a commutative diagram of spectra with $\Sigma_{n-1}$-action
    \[
    \begin{tikzcd}
    \mathrm{Res}^{\Sigma_n}_{\Sigma_{n-1}} \sQ(n)  \ar[r] \ar[d, "\wr"] &  \mathrm{Res}^{\Sigma_n}_{\Sigma_{n-1}} \sQ(n) \ar[d, "\wr"] \\
         \sQ(n-1) \otimes \sQ(1) \ar[r] &  \coCom(n-1) \otimes \coCom(1).
    \end{tikzcd}
    \]
    We may take as our inductive hypothesis that the bottom horizontal map is an isomorphism, so that the top horizontal map is an isomorphism as well.
    Since $\mathrm{Res}^{\Sigma_n}_{\Sigma_{n-1}}$ is a conservative functor, this completes the proof.
\end{proof}

\subsection{Derivatives of the identity functor in pointed spaces}\label{sec: derivatives-identity-spaces}

Before giving the proof of \cref{introthm: characterization}, we will use the proposition from the previous section to prove that $\partial_*{\id_{\Spc_*}}$ is the spectral Lie operad; see \cref{ex: der-identity-pointed-spaces} above for how this proposition relates to other versions of this result in the literature.

\begin{proposition} \label{prop: derivatives-id-Lie}
    There is an equivalence $\partial_*{\id_{\Spc_*}} \simeq \LL$ of operads.
\end{proposition}
\begin{proof}
    The functor $\Sigma^{\infty} \colon \Spc_* \to \Sp$ is strong symmetric monoidal from $\wedge$ to $\otimes$.
    Recall from \cref{prop: differentiating-sym-mon-cats} that the Goodwillie transform
    \[
    \maLMod \colon \diff \to \diff\colon \quad \sC \mapsto \LMod_{\partial_*{\id_\sC}}(\SSeqnu(\Sp(\sC)))
    \]
    preserves symmetric monoidal $\infty$-categories and symmetric monoidal functors.
    By \cref{ex: derivative-tensor-product-sym-mon,ex: derivative-smash-product-sym-mon}, $\maLMod$ sends the tensor product on $\Sp$ to the Day convolution tensor product on $\SSeqnu(\Sp)$ and the smash product on $\Spc_*$ to the smash product on $\LMod_{\partial_*{\id_{\Spc_*}}}(\SSeqnu(\Sp))$.
    It follows that the functor $\maLMod(\Sigma^\infty)$, which is given by the bar construction
    \[
    B(\partial_*{\Sigma^\infty}, \partial_*{\id_{\Spc_*}}, -) \colon \LMod_{\partial_*{\id_{\Spc_*}}} \to\SSeqnu(\Sp),
    \]
    acquires a symmetric monoidal structure from the smash product to the Day convolution tensor product.
    Since $\partial_*{\Sigma^\infty} \simeq \unit$, it has to be a trivial right $\partial_*{\id_{\Spc_*}}$-module for degree reasons, so that we have an equivalence of functors
    \[
    B(\partial_*{\Sigma^\infty}, \partial_*{\id_{\Spc_*}}, -) \simeq \cot_{\partial_*{\id_{\Spc_*}}}.
    \]
    By precomposing this functor with the Koszul duality equivalence
    \[
    \prim_{B \partial_*{\id_{\Spc_*}}} \colon \LcoMod_{B \partial_*{\id_{\Spc_*}}} \xrightarrow{\sim} \LMod_{\partial_*\id_{\Spc_*}}
    \]
    of \cref{prop: koszul-duality-left-modules} and using that $\cot_{\partial_*{\id_{\Spc_*}}} \circ \prim_{B \partial_*{\id_{\Spc_*}}} \simeq \forget_{B \partial_*{\id_{\Spc_*}}}$,
    we find that this forgetful functor acquires a nonunital symmetric monoidal structure from $\wedge$ to $\otimes$.
    By \cref{prop: characterization-cocom}, we then have that 
    $B \partial_*{\id_{\Spc_*}} \simeq \coCom$, so that $\partial_*{\id_{\Spc_*}} \simeq \LL$.
\end{proof}

\begin{corollary} \label{cor: derivatives-identity-pointed-spaces}
    There is an equivalence $\partial_*(\Sigma^\infty \Omega^\infty) \simeq \coCom$ of cooperads.
\end{corollary}
\begin{proof}
    This follows by combining the previous proposition with \cref{prop: Koszul-dual-derivative-id}.
\end{proof}

\subsection{The free Lie algebra functor} \label{sec: free-Lie-algebra}

We will show in this section that the free spectral Lie algebra functor is symmetric monoidal with respect to the tensor product of spectra and the smash product of Lie algebras.
We will do so by first proving an analogous statement in the Morita category.

\begin{definition}
    Let $\sO$ and $\sP$ be operads in spectra and let $n \geq 1$. 
    We write $\MMor^{\mathrm{L}}((\Sp, \sO)^{\times n}, (\Sp, \sP))$ for the full subcategory of
    \[
    \MMor((\Sp, \sO)^{\times n}, (\Sp, \sP)) \simeq \BMod_{(\sP, \sO)}(\mSSeqnu{n})
    \]
    spanned by the $(\sP, \sO)$-bimodules for which the underlying left $\sP$-module is free on an $n$-fold symmetric sequence concentrated in multidegree $(1, \ldots, 1)$.
\end{definition}

\begin{remark}
Note that if $M$ is a $(\sP, \sO)$-bimodule in $\mSSeqnu{n}$, then it is a condition and not extra structure for the underlying left $\sP$-module to be free on an $n$-fold symmetric sequence concentrated in multidegree $(1, \ldots, 1)$.
Indeed, the map of $n$-fold symmetric sequences $M_{(1, \ldots, 1)} \to M$ given by the inclusion of the arity $(1, \ldots, 1)$ term induces a map of left $\sP$-modules
\[
\sP \circ M_{(1, \ldots, 1)} \to M,
\]
and this map is an equivalence if and only if $M$ is free on an $n$-fold symmetric sequence of multidegree $(1, \ldots, 1)$.
\end{remark}

\begin{example}
    Recall we have an equivalence of $\infty$-categories 
    \[
    \MMor((\Sp, \unit)^{\times n}, (\Sp, \unit)) \simeq \mSSeqnu{n}.
    \]
    An $n$-fold symmetric sequence lies in the full subcategory $\MMor^{\mathrm{L}}((\Sp, \unit)^{\times n}, (\Sp, \unit))$
    if and only if it is concentrated in multidegree $(1, \ldots, 1)$.
\end{example}

\begin{lemma}\label{lem: postcomposition-cot-on-Mor}
    Let $\sO$ be an operad in spectra and let $n \geq 1$.
    Then postcomposition with the trivial right $\sO$-module $\unit$ induces an equivalence
    \[
    \begin{tikzcd}[sep = large]
    \MMor^{\mathrm{L}}((\Sp, \unit)^{\times n}, (\Sp, \sO)) \ar[r, "\unit \circ_\sO -", "\sim"'] & \MMor^{\mathrm{L}}((\Sp, \unit)^{\times n}, (\Sp, \unit)).
    \end{tikzcd}
    \]
\end{lemma}
\begin{proof}
    The functor $\unit \circ_\sO -$ is clearly essentially surjective.
    To see that it is fully faithful, suppose that we are given
    \[
    M, N \in \MMor^{\mathrm{L}}((\Sp, \unit)^{\times n}, (\Sp, \sO)).
    \]
    Then $M \simeq \sO \circ M_{(1, \ldots, 1)}$ and $N \simeq \sO \circ N_{(1, \ldots, 1)}$ by assumption.
    We have
    \[
    \Map_{\LMod_\sO}(M, N) \simeq \Map_{\SSeq}(M_{(1, \ldots, 1)}, \sO \circ N_{(1, \ldots 1)}) \simeq \Map_\Sp(M_{(1, \ldots, 1)}, N_{(1, \ldots, 1)}).
    \]
    Here, we wrote $\Map_{\LMod_\sO}$ for the mapping space in $\MMor^{\mathrm{L}}((\Sp, \unit)^{\times n}, (\Sp, \sO))$ and $\Map_{\SSeq}$ for the mapping space in $\SSeqnu(\Sp^{\times n}, \Sp)$.
    The first equivalence follows by adjunction, and the second equivalence follows since a morphism of symmetric sequences from $M_{(1, \ldots, 1)}$ to $\sO \circ N_{(1, \ldots, 1)}$ has to factor through the arity $(1, \ldots, 1)$-term of the target, which is given by $N_{(1, \ldots, 1)}$.
    So we find that the space of maps from $M$ to $N$ is equivalent to the space of maps from $\unit \circ_\sO M$ to $\unit \circ_\sO N$, and this equivalence is clearly given by the functor $\unit \circ_\sO -$.
\end{proof}

\begin{definition}
We write $\CMonnu(\MMor)$ for the $(\infty, 2)$-precategory of nonunital commutative monoids in $\MMor$.
In other words, this is the full subcategory of $\Fun(\Surj_*, \MMor)$ spanned by the functors that satisfy the Segal condition.
\end{definition}

\begin{remark}
Every $M \in \CMonnu(\MMor)$ has an underlying object $(\sC, \sO)$ given by evaluating the corresponding functor $M \colon \Surj_* \to \MMor$ at the pointed set $\langle 1 \rangle$ with one non-basepoint element.
The unique surjection $\langle n \rangle \to \langle 1\rangle$ for which the preimage of the basepoint is the basepoint induces a multiplication map
\[
\otimes_\sO^n \colon (\sC, \sO)^{\times n} \to (\sC, \sO)
\]
in $\MMor$.
These multiplication maps are suitably associative and commutative.
We will often write $(\sC, \sO, \otimes_\sO)$ instead of $M$, or just $(\sO, \otimes_\sO)$ in case $\sC = \Sp$.
\end{remark}

Given  $(\sO, \otimes_\sO), (\sP, \otimes_\sP) \in \CMonnu(\MMor)$ we write
\[
\CMonnu(\MMor)^{\mathrm{L}}((\sO, \otimes_\sO), (\sP, \otimes_\sP))
\]
for the category of maps of nonunital commutative monoids such that the underlying morphism lies in $\MMor^{\mathrm{L}}((\Sp, \sO), (\Sp, \sP))$.

\begin{lemma}\label{lem: postcomposition-cot-on-Mor-monoidal}
    Suppose that we are given $(\sO, \otimes_\sO), (\unit, \otimes_\unit) \in \CMonnu(\MMor)$ such that the following conditions are satisfied:
    \begin{itemize}
        \item[\upshape{(}1\upshape{)}] For each $n \geq 1$, the multiplication maps $\otimes_\sO^n$ and $\otimes_\unit^n$ lie in the categories
        \[
        \MMor^{\mathrm{L}}((\Sp, \sO)^{\times n}, (\Sp, \sO)) \qquad \text{and} \qquad \MMor^{\mathrm{L}}((\Sp, \unit)^{\times n}, (\Sp, \unit)),
        \]
        respectively.
        \item[\upshape{(2)}] The trivial right $\sO$-module $\unit$, considered as a morphism 
        \[
        \unit \colon (\Sp, \sO) \to (\Sp, \unit)
        \] 
        in $\MMor$, admits the structure of a map of nonunital commutative monoids $(\sO, \otimes_\sO) \to (\unit, \otimes_\unit)$.
    \end{itemize}
    Then postcomposition with the morphism from \upshape{(2)} induces an equivalence
    \[
    \begin{tikzcd}[sep = large]
        \CMonnu(\MMor)^{\mathrm{L}}((\unit, \otimes_\unit), (\sO, \otimes_\sO)) \ar[r, "\unit \circ_\sO -", "\sim"'] & \CMonnu(\MMor)^{\mathrm{L}}((\unit, \otimes_\unit), (\unit, \otimes_\unit)).
    \end{tikzcd}
    \]
\end{lemma}
\begin{proof}
    By \cite[Lemma 6.4]{GepnerHaugsengNikolaus}, the $\infty$-category of morphisms $\CMonnu(\MMor)((\unit, \otimes_\unit), (\sO, \otimes_\sO))$ is equivalent to the limit of the functor
    \[
    \begin{tikzcd}[sep = 3em]
       K \colon \Tw(\Surj_*) \ar[r] & \Surj_*^{\op} \times \Surj_* \ar[rr, "{(\unit, \otimes_{\unit}) \times (\sO, \otimes_{\sO})}"] & & (\MMor)^{\op} \times \MMor \ar[r, "\MMor{(-, -)}"] & \Cat_\infty,
    \end{tikzcd}
    \]
    where $\Tw(\Surj_*)$ denotes the twisted arrow category of $\Surj_*$, and the first map is the projection onto source and target.
    In other words, this category of morphisms can be written as a limit over $\infty$-categories of the form
    \begin{equation} \label{eq: morita-end}
    \MMor((\Sp, \unit)^{\times m}, (\Sp, \sO)^{\times n})
    \end{equation}
    for $m, n \geq 1$.
    We claim that the full subcategory $\CMonnu(\MMor)^{\mathrm{L}}((\unit, \otimes_\unit), (\sO, \otimes_\sO))$ can be represented as the limit of a subfunctor of $K$, which is pointwise given by the full subcategory 
    \[
    \MMor^{\mathrm{L}}((\Sp, \unit)^{\times m}, (\Sp, \sO))^{\times n}
    \]
    of \cref{eq: morita-end}.
    We first verify that this defines a subfunctor.
    To do so, it suffices to show that for any surjection $\alpha \colon \langle n \rangle \to \langle \ell \rangle$, the induced functor
    \[
    \alpha_* \colon \MMor((\Sp, \unit)^{\times m}, (\Sp, \sO)^{\times n}) \to \MMor((\Sp, \unit)^{\times m}, (\Sp, \sO)^{\times \ell})
    \]
    maps the full subcategory $\MMor^{\mathrm{L}}((\Sp, \unit)^{\times m}, (\Sp, \sO))^{\times n}$ into $\MMor^{\mathrm{L}}((\Sp, \unit)^{\times m}, (\Sp, \sO))^{\times \ell}$,
    and also that the analogous statement holds for precomposition by any surjection $\beta \colon \langle k \rangle \to \langle m \rangle$.
    We only prove the first of these statements, since the second one follows by a similar argument.
    It is moreover enough to verify the first statement for $\ell = 1$. 
    We will now proceed to do so.
    
    Suppose that $M \in \MMor^{\mathrm{L}}((\Sp, \unit)^{\times m}, (\Sp, \sO))^{\times n}$.
    We have $M = (M_1, \ldots, M_n)$, and $M_i \simeq \sO \circ X_i$ as left $\sO$-modules with $X_i$ concentrated in multidegree $(1, \ldots, 1)$, where $1$ occurs $m$ times.
    By assumption, we also have $\otimes_\sO^{\times n} = \sO \circ Y$ as left $\sO$-modules with $Y$ concentrated in multidegree $(1, \ldots, 1)$, where $1$ occurs $n$ times.
    We need to show that the underlying left $\sO$-module of $\otimes_\sO^n \circ_{\sO^{\times n}} M$ is free on a symmetric sequence concentrated in multidegree $(1, \ldots, 1)$, where $1$ occurs $mn$ times.
    But as left $\sO$-modules, we have equivalences
    \[
    \otimes_\sO^n \circ_{\sO^{\times n}} M \simeq \otimes_\sO^n \circ_{\sO^{\times n}} \sO^{\times n} \circ (X_1, \ldots, X_n) \simeq \sO \circ Y \circ (X_1, \ldots, X_n),
    \]
    and $Y \circ (X_1, \ldots, X_n)$ is concentrated in multidegree $(1, \ldots, 1)$ by \cref{prop: multi-composition-product}.
    This shows that we indeed get a subfunctor.

    It now follows easily that $\CMonnu(\MMor)^{\mathrm{L}}((\unit, \otimes_\unit), (\sO, \otimes_\sO))$ is the limit of this subfunctor:
    if $M$ is an element of this limit, then $M$ in particular is an element of $\lim K$, and hence defines a map of nonunital commutative monoids.
    Moreover, its underlying morphism lies in $\MMor^{\mathrm{L}}((\Sp, \unit), (\Sp, \sO))$ by definition of the subfunctor.
    On the other hand, if $M \in \CMonnu(\MMor)^{\mathrm{L}}((\unit, \otimes_\unit), (\sO, \otimes_\sO))$, then we can represent $M$ as an element of $\lim K$ such that the component in $K(\id_{\langle 1 \rangle})$, which corresponds to the underlying morphism, lies in $\MMor^{\mathrm{L}}((\Sp, \unit), (\Sp, \sO))$, i.e. in the subfunctor.
    But then it follows that all components of $M$ in $\lim K$ take values in the subfunctor, since these components are obtained by composing the underlying morphism of $M$ with $\otimes_\sO^m$ and $\otimes_\unit^n$ for $m, n \geq 1$.
    Hence $M$ lies in the limit of the subfunctor.

    The category $\CMonnu(\MMor)^{\mathrm{L}}((\unit, \otimes_\unit), (\unit, \otimes_\unit))$ can be written as a limit over $\Tw(\Surj_*)$ in the same way.
    To complete the proof, it therefore suffices to show that postcomposition with the trivial right $\sO$-module $\unit$ induces an equivalence
    \[
    \begin{tikzcd}[sep=large]
    \MMor^{\mathrm{L}}((\Sp, \unit)^{\times m}, (\Sp, \sO))^{\times n} \ar[r, "(\unit \circ_\sO -)^{\times n}"] & 
    \MMor^{\mathrm{L}}((\Sp, \unit)^{\times m}, (\Sp, \unit))^{\times n}.
    \end{tikzcd}
    \]
    This follows from \cref{lem: postcomposition-cot-on-Mor}.
\end{proof}

We will now use this result to prove that the free Lie algebra functor is symmetric monoidal.

\begin{proposition} \label{prop: free-lie-sym-mon}
    The functor $\free_\LL \colon \Sp \to \Lie(\Sp)$ admits a nonunital symmetric monoidal structure with respect to the tensor product on $\Sp$ and the smash product on $\Lie(\Sp)$.
\end{proposition}
\begin{proof}
    The functor $\Sigma^\infty \colon \Spc_* \to \Sp$ is nonunital symmetric monoidal with respect to the smash product on $\Spc_*$ and the tensor product on $\Sp$.
    Taking derivatives of this functor yields a morphism
    \[
    \partial_*(\Sigma^\infty) \colon (\LL, \partial_*(-\wedge-)) \to (\unit, \partial_*(- \otimes -))
    \]
    of nonunital commutative monoids in $\MMor$ by \cref{prop: partial-mor-preserves-products}.
    We claim that this morphism satisfies the conditions of \cref{lem: postcomposition-cot-on-Mor-monoidal}.
    Indeed, by \cref{rem: derivative-colim-pres-multivar-functor} it follows from the fact that the smash product commutes with colimits in each variable separately that $\partial_*(\wedge^n)$ lies in $\MMor^{\mathrm{L}}((\Sp, \LL)^{\times n}, (\Sp, \LL))$.
    Here $\wedge^n \colon \Spc_*^{\times n} \to \Spc_*$ denotes the $n$-fold smash product functor.
    It is also clear that $\partial_*(\otimes^n)$ lies in $\MMor^{\mathrm{L}}((\Sp, \unit)^{\times n}, (\Sp, \unit))$,
    as $\otimes^n \colon \Sp^{\times n} \to \Sp$ is multilinear.
    Since $\partial_*(\Sigma^\infty) \simeq \unit$, we see that all conditions are met.

    As a consequence of this lemma, there is a unique morphism
    \[
    F \in \CMonnu(\MMor)^{\mathrm{L}}((\unit, \partial_*(- \otimes -)), (\LL, \partial_*(- \wedge -)))
    \]
    such that $\unit \circ_\LL F \simeq \unit$, the identity of $(\unit, \partial_*(- \otimes -))$.
    It follows from \cref{lem: postcomposition-cot-on-Mor} that the underlying morphism of $F$ must be the free left $\LL$-module $\LL$.
    We now apply $\reAlg$ to $F$.
    Since the underlying module of $F$ is $\LL$, the result is a nonunital symmetric monoidal structure on the functor
    \[
    \free_\LL \colon \Sp \to \Lie(\Sp).
    \]
    It follows from \cref{ex: derivative-tensor-product-sym-mon,ex: derivative-smash-product-sym-mon} that the induced nonunital symmetric monoidal structure on $\Sp$ is the tensor product and on $\Lie(\Sp)$ is the smash product.
    This completes the proof.
\end{proof}

We will also need the following variation of \cref{lem: postcomposition-cot-on-Mor-monoidal} in the next section.
It is proved in exactly the same way.

\begin{lemma} \label{lem: precomposition-free-on-Mor-mon}
    Suppose that we are given $(\sO, \otimes_\sO), (\unit, \otimes_\unit) \in \CMonnu(\MMor)$ such that the following conditions are satisfied:
    \begin{itemize}
        \item[\upshape{(}1\upshape{)}] For each $n \geq 1$, the multiplication maps $\otimes_\sO^n$ and $\otimes_\unit^n$ lie in the categories
        \[
        \MMor^{\mathrm{L}}((\Sp, \sO)^{\times n}, (\Sp, \sO)) \qquad \text{and} \qquad \MMor^{\mathrm{L}}((\Sp, \unit)^{\times n}, (\Sp, \unit)),
        \]
        respectively.
        \item[\upshape{(2)}] The free left $\sO$-module $\sO$, considered as a morphism 
        \[
        \sO \colon (\Sp, \unit) \to (\Sp, \sO)
        \] 
        in $\MMor$, admits the structure of a map of nonunital commutative monoids $(\unit, \otimes_\unit) \to (\sO, \otimes_\sO)$.
    \end{itemize}
    Then precomposition with the morphism from \upshape{(2)} induces an equivalence
    \[
    \begin{tikzcd}[sep = large]
        \CMonnu(\MMor)^{\mathrm{L}}((\sO, \otimes_\sO), (\unit, \otimes_\unit)) \ar[r, "- \circ_\sO \sO", "\sim"'] & \CMonnu(\MMor)^{\mathrm{L}}((\unit, \otimes_\unit), (\unit, \otimes_\unit)).
    \end{tikzcd}
    \]
\end{lemma}

\subsection{Proof of the characterization} \label{sec: proof-characterization}

We now prove our characterization of the spectral Lie operad:

\begin{theorem}[\cref{introthm: characterization}] \label{thm: characterization}
Let $\sO$ be a nonunital and reduced operad in spectra. Then the following are equivalent:
\begin{enumerate}[\upshape{(}\arabic*\upshape{)}]
    \item The operad $\sO$ is equivalent to the spectral Lie operad $\LL$.
    \item The cooperad $B\sO$ is equivalent to the nonunital cocommutative cooperad $\mathbf{Com}^{\vee}$.
    \item
    The free $\sO$-algebra functor 
    \[
    \free_{\sO} \colon \Sp \to \Alg_{\sO}(\Sp)
    \] 
    admits a nonunital symmetric monoidal structure with respect to the tensor product on $\Sp$ and the smash product on $\Alg_{\sO}(\Sp)$.
    \item The functor
    \[
        \cot_{\sO} \colon \Alg_{\sO}(\Sp) \to \Sp
    \]
    admits a nonunital symmetric monoidal structure with respect to the smash product on $\Alg_{\sO}(\Sp)$ and the tensor product on $\Sp$.
    \item The functor
    \[
    \forget \colon \coAlg^{\mathrm{dp, nil}}_{B\sO}(\Sp) \to \Sp
    \]
    admits a nonunital symmetric monoidal structure with respect to the smash product on $\coAlg^{\mathrm{dp, nil}}_{B\sO}(\Sp)$ and the tensor product on $\Sp$.
\end{enumerate}
\end{theorem}
\begin{proof}
	$(1) \Leftrightarrow (2):$ This holds by definition.
	
	$(1) \Rightarrow (3):$ This is \cref{prop: free-lie-sym-mon}.
	
	$(3) \Rightarrow (4):$ First of all, it follows from \cref{prop: smash-product-T-algebras-associative} that the smash product defines an honest nonunital symmetric monoidal structure on $\Alg_{\sO}(\Sp)$.
    We also find that the smash product on $\Alg_{\sO}(\Sp)$ preserves colimits in each variable, since this holds for the tensor product of spectra, $\free_\sO$ is symmetric monoidal, and every $\sO$-algebra can be written as a geometric realization of free $\sO$-algebras.
    
    Taking derivatives, we obtain a morphism
    \[
    \partial_*(\free_\sO) \colon (\unit, \partial_*(- \otimes -)) \to (\sO, \partial_*(- \wedge -))
    \]
    of nonunital commutative monoids in the Morita category.
    Here, we have used the equivalence $\partial_*{\id_{\Alg_\sO}} \simeq \sO$ from \cref{ex: der-identity-O-algebras} to identify the right hand side.
    We claim this morphism satisfies the conditions of \cref{lem: precomposition-free-on-Mor-mon}.
    Since the smash product on $\Alg_\sO(\Sp)$ commutes with colimits in each variable, it follows from \cref{rem: derivative-colim-pres-multivar-functor} that for each $n \geq 1$, $\partial_*(\wedge^n)$ is free as a left $\sO$-module on a symmetric sequence concentrated in multidegree $(1, \ldots, 1)$.
    It is also clear that $\partial_*(\otimes^n)$ is concentrated in this multidegree.
    Finally, it follows from \cref{prop: derivative-left-adjoint} that $\partial_*(\free_\sO)$ is the free left $\sO$-module $\sO$.

    As a consequence of \cref{lem: precomposition-free-on-Mor-mon}, there is a unique morphism
    \[
    K \in \CMonnu(\MMor^{\mathrm{L}})((\sO, \partial_*(- \wedge -)), (\unit, \partial_*(- \otimes -)))
    \]
    such that $K \circ_\sO \sO \simeq \unit$, the identity of $(\unit, \partial_*(- \otimes -))$.
    The underlying morphism of $K$ must be the trivial right $\sO$-module $\unit$.
    We now apply $\reAlg$ to $K$.
    Since the underlying module of $K$ is $\unit$, the result is a nonunital symmetric monoidal structure on the functor
    \[
    \cot_\sO \colon \Alg_\sO(\Sp) \to \Sp.
    \]
    It follows from \cref{ex: derivative-tensor-product-sym-mon,ex: derivative-smash-product-sym-mon} that the induced nonunital symmetric monoidal structure on $\Sp$ is the tensor product and on $\Alg_\sO(\Sp)$ is the smash product.

	$(4) \Rightarrow (5):$ Since we have an equivalence of comonads $\cot_\sO \circ \triv_\sO \simeq \Sym_{B\sO}$, this follows from the implication $(1) \Rightarrow (5)$ of \cref{prop: lax-sm-comonad-over-sp-aug}.

    $(5) \Rightarrow (2):$ By \cref{prop: lax-sm-comonad-over-sp-aug}, the assumption is equivalent to the statement that $\Sym_{B\sO} \colon \Sp \to \Sp$ admits the structure of a lax symmetric monoidal comonad for $\otimesaug$ such that the composite
    \[
    \Sym_{B\sO}(X) \otimesaug \Sym_{B\sO}(Y) \to \Sym_{B\sO}(X \otimesaug Y) \to \Sym_{B\sO}(X \times Y)
    \]
    is an equivalence.
    Applying the Goodwillie transform $\maLMod$, we find that the same holds for the comonad $\Sym_{B\sO} \colon \SSeqnu(\Sp) \to \SSeqnu(\Sp)$.
    By another appeal to \cref{prop: lax-sm-comonad-over-sp-aug}, this implies that the forgetful functor
    \[
    \forget_{B\sO} \colon \LcoMod_{B\sO} \to \SSeqnu(\Sp)
    \]
    is nonunital symmetric monoidal with respect to the smash product and the tensor product by \cref{prop: lax-sm-comonad-over-sp-aug}.
    It follows that $B\sO \simeq \coCom$ by \cref{prop: characterization-cocom}.
\end{proof}

\begin{remark}
    It follows from \cref{prop: lax-sm-comonad-over-sp-aug} that points (4) and (5) from \cref{thm: characterization} can be replaced by
    \begin{itemize}
        \item[(4*)]
        The functor
        \[
            (\mathrm{cot}_{\sO})_+\colon \Alg_{\sO}(\Sp) \to \Sp_{\Sph//\Sph} \colon X \mapsto \Sph \oplus \mathrm{cot}_{\sO}(X)
        \]
        admits a symmetric monoidal structure with respect to the cartesian product on $\Alg_{\sO}(\Sp)$ and the tensor product on $\Sp$.
        \item[(5*)]
         The functor
         \[
         \forget \colon \coAlg^{\mathrm{dp, nil}}_{B\sO}(\Sp) \to \Sp
        \]
        admits a  symmetric monoidal structure with respect to the cartesian product on $\coAlg^{\mathrm{dp, nil}}_{B\sO}(\Sp)$ and the augmented tensor product on $\Sp$.
    \end{itemize}
    Moreover, (4*) is equivalent to the functor
    \[
    \indec_\sO \colon \Alg_\sO(\Sp) \to \coAlg^{\mathrm{dp, nil}}_{B\sO}(\Sp)
    \]
    preserving cartesian products.
    Also note that (5*) implies by \cref{cor: exponential-property} that $\Sym_{B\sO}$ is symmetric monoidal from $\oplus$ to $\otimesaug$.
\end{remark}

\appendix

\section{Smash products} \label{sec: smash}

Let $(\sC, \otimes)$ be a symmetric monoidal $\infty$-category which is pointed and admits finite colimits, and suppose that the zero object, which we denote by $\ast$, is the unit of the symmetric monoidal structure.
For $X, Y \in \sC$, we define their \emph{smash product} $X \wedge Y$ via the cofiber sequence
\[
X \sqcup Y \to X \otimes Y \to X \wedge Y,
\]
where the first arrow denotes the map from the coproduct to the monoidal product with components
\[
X \simeq X \otimes \ast \rightarrow X \otimes Y \leftarrow \ast \otimes Y \simeq Y.
\]

In general, the smash product does not extend to a symmetric monoidal structure: if the monoidal product of $\sC$ doesn't commute with enough colimits, the smash product can fail to be associative.
The relevant type of colimits are as follows.

\begin{definition}
	Let $I$ be a small $\infty$-category.
	We say $I$ is \emph{weakly contractible} if the space $\lvert I \rvert \in \Spc$ is contractible, where $\lvert{}-\rvert \colon \Cat \to \Spc$ denotes the left adjoint to the inclusion $\Spc \hookrightarrow \Cat$. A \emph{weakly contractible colimit} is a colimit diagram that is indexed by a small weakly contractible $\infty$-category.
\end{definition}

The main result of this appendix is that if the tensor product commutes with weakly contractible colimits in both variables, the smash product extends to an essentially unique nonunital symmetric monoidal structure on $\sC$.
In proving this, it is convenient to weaken the assumption that the zero object $\ast$ is the unit for $\otimes$: we will only require that $\ast \otimes \ast = \ast$. 
The smash product $X \wedge Y$ is then defined as the cofiber of the map
\[
X \otimes \ast \sqcup \ast \otimes Y \to X \otimes Y.
\]
The precise statement we will prove is as follows:

\begin{proposition}\label{prop: smash-monoidal-structure}
	Let $(\sC, \otimes)$ be a nonunital symmetric monoidal $\infty$-category which admits finite colimits, is pointed, and satisfies $\ast \otimes \ast = \ast$, where $\ast$ is the zero object of $\sC$.
		Suppose that 
	\[
	- \otimes - \colon \sC \times \sC \to \sC
	\]
	preserves finite weakly contractible colimits in both variables separately.
	Then there exists an essentially unique nonunital symmetric monoidal structure on $\sC$ for which 
	\begin{enumerate}[\upshape{(}\arabic*\upshape{)}]
		\item the monoidal product is given by the smash product, and
		\item the canonical morphism $X \otimes Y \to X \wedge Y$ extends to a lax symmetric monoidal functor $\rho \colon (\sC, \wedge) \to (\sC, \otimes)$ which lifts the identity functor of $\sC$.
	\end{enumerate}
\end{proposition}

The uniqueness of the smash product symmetric monoidal structure will follow from the fact that it satisfies a universal property.
Before stating this universal property, we first make the following observation.

\begin{proposition} \label{prop: smash-preserves-colimits}
    Let $(\sC, \otimes)$ be a nonunital symmetric monoidal $\infty$-category as in the previous proposition.
    Then the smash product $\wedge \colon \sC \times \sC \to \sC$ preserves finite colimits in both variables separately.
\end{proposition}
\begin{proof}
    Since we can write
    \[
    X \wedge Y \simeq \mathrm{cofib}(X \otimes \ast \sqcup \ast \otimes Y \to X \otimes Y),
    \]
    and since both $\sqcup$ and $\otimes$ preserve weakly contractible colimits in both variables, it follows that $\wedge$ also preserves such colimits in both variables.
    As $\ast \otimes \ast \simeq \ast$ by assumption we also find that $\ast \wedge Y \simeq \ast \simeq X \wedge \ast$.
    The statement now follows, since a functor between finite cocomplete pointed categories preserves finite colimits if and only if it preserves finite weakly contractible colimits and the zero object.
\end{proof}

The universal property of the smash product is as follows.

\begin{proposition}\label{prop: smash-product-universal-property}
    Let $(\sC, \otimes_\sC)$ be a nonunital symmetric monoidal $\infty$-category as in \cref{prop: smash-monoidal-structure}.
    For every pointed nonunital symmetric monoidal $\infty$-category $(\sD, \otimes_\sD)$ for which the tensor product commutes with all finite colimits in both variables separately, postcomposition with $\rho \colon (\sC, \wedge) \to (\sC, \otimes_\sC)$ induces an equivalence
    \[
    \Fun_*^{\mathrm{lax}}((\sD, \otimes_\sD), (\sC, \wedge)) \xrightarrow{\sim} \Fun_*^{\mathrm{lax}}((\sD, \otimes_\sD), (\sC, \otimes_\sC)),
    \]
    where $\Fun_*^{\mathrm{lax}}$ denotes the $\infty$-category of lax symmetric monoidal functors whose underlying functor preserves the zero object.
\end{proposition}

The construction of the smash product symmetric monoidal structure we will give here is essentially the one contained in the unpublished note \cite{raksit2018smash}.
Since our assumptions are slightly different, we have decided to reproduce it here.
In what follows, we let $(\sC, \otimes)$ be a nonunital symmetric monoidal $\infty$-category satisfying the assumptions from the proposition.
It will become clear in the course of the proof why we only get a nonunital symmetric monoidal structure, even if $\sC$ admits a unit; see \cref{rem: non-unital-smash} below.

We start with a number of preparatory lemmas.
Let $[1] = (0 \to 1)$ denote the free morphism.
We endow it with the minimum symmetric monoidal structure.

\begin{lemma}
	The arrow category $\Ar(\sC) = \Fun([1], \sC)$ becomes a nonunital symmetric monoidal $\infty$-category under Day convolution.
\end{lemma}
\begin{proof}
	By \cite[Proposition 2.2.6.16]{HA}, it suffices to show that for $x \in \{0, 1\}$, for every $n \geq 1$, and for every sequence of arrows $(f_i \colon [1] \to \sC)_{i = 1}^n$, the composite
	\[
	\begin{tikzcd}
	{[1]}^{\times n} \times_{[1]} {[1]}_{/x} \ar[r] & {[1]}^{\times n} \ar[r, "(f_i)"] & \sC^{\times n} \ar[r, "\otimes"] & \sC
	\end{tikzcd}
	\]
	admits a colimit, and that the tensor product $\otimes$ on $\sC$ commutes with colimits indexed by ${[1]}^{\times n} \times_{[1]} {[1]}_{/x}$ in both variables separately.
	Here the functor $[1]^{\times n} \to [1]$ used to define the pullback is given by taking the minimum.
    (It is assumed in the cited proposition that the tensor product commutes with all finite colimits; it is however clear from the proof that it only has to commute with colimits of these particular shapes.)
	
	For $x=0$, the category ${[1]}^{\times n} \times_{[1]} {[1]}_{/0}$ is the full subcategory of $[1]^{\times n}$ consisting of those $n$-tuples $(y_j)$ for which at least one of the $y_j$ is equal to $0$.
	In other words, it is a punctured cube. 
	As this is a finite weakly contractible category, it follows from our assumptions that $\sC$ admits and $\otimes$ preserves colimits indexed by it.
	For $x=1$, the indexing category admits a final object so that both conditions are trivially satisfied.
\end{proof}
We write $\otimesday$ for the nonunital Day convolution symmetric monoidal structure on $\Ar(\sC)$.

\begin{remark} \label{rem: non-unital-smash}
	Even if we assume $(\sC, \otimes)$ admits a unit, the Day convolution symmetric monoidal structure defined in the previous lemma in general does not have one.
	Indeed, a unit would be guaranteed to exist in this case if $\sC$ admits and $\otimes$ commutes with colimits indexed by $[1]^{\times 0} \times_{[1]} [1]_{/x}$ for $x = 0, 1$.
	The map $[1]^{\times 0} \to [1]$ used to define the pullback picks out the unit of the minimum symmetric monoidal structure, which is $1$.
	Therefore, this pullback is the empty category for $x=0$, which is in particular not weakly contractible.
\end{remark}

\begin{example} \label{ex: formula-day-convolution-arrow}
	Suppose that $f \colon A \to B$ and $g \colon X \to Y$ are arrows in $\sC$.
	Then their Day convolution is given by the canonical morphism
	\[
	f \otimesday g \colon A \otimes Y \sqcup_{A \otimes X} B \otimes X \to B \otimes Y.
	\]
	In particular, $(\ast \to X) \otimesday (\ast \to Y)$ is the morphism $X \otimes \ast \sqcup \ast \otimes Y \to X \otimes Y$ whose cofiber is $X \wedge Y$.
\end{example}

Observe that the functor $\mathrm{cofib} \colon \Ar(\sC) \to \sC$ exhibits $\sC$ as a reflective localization of $\Ar(\sC)$.
Its fully faithful right adjoint is given by the functor that sends $X \in \sC$ to the arrow $\ast \to X$.

\begin{lemma}
	The localization $\mathrm{cofib} \colon \Ar(\sC) \to \sC$ is compatible with the Day convolution symmetric monoidal structure.
\end{lemma}
\begin{proof}
	It suffices to show that for every pair of morphisms $f \colon A \to B$ and $g \colon X \to Y$, the morphism $f \otimesday g \to \mathrm{cofib}(f) \otimesday g$ is sent to an isomorphism by $\mathrm{cofib}$.
	This is equivalent to showing that the commutative square
	\[
	\begin{tikzcd}
		A \otimes Y \sqcup_{A \otimes X} B \otimes X \ar[rr, "f \otimesday g"] \ar[d] & & B \otimes Y \ar[d] \\
		\ast \otimes Y \sqcup_{\ast \otimes X} \cof(f) \otimes X \ar[rr, "\mathrm{cofib}(f) \otimesday g"] & & \cof(f) \otimes Y 
	\end{tikzcd}
	\]
	induces an isomorphism on horizontal cofibers.
	A simple diagram chase using that $\otimes$ commutes with weakly contractible colimits in both variables shows that this is a consequence of the well-known fact that the total cofiber of the square
	\[
	\begin{tikzcd}
		A \otimes X \ar[r] \ar[d] & B \otimes X \ar[d] \\
		A \otimes Y \ar[r] & B \otimes Y
	\end{tikzcd}
	\]
	can be computed in the following two ways:
	either as the cofiber of the canonical map from the pushout to the bottom right corner, which corresponds to the cofiber of the top horizontal morphism in the previous square,
	or as cofiber of the induced map on horizontal cofibers, which corresponds to the cofiber of the bottom horizontal morphism in the previous square.
\end{proof}

We now prove that the smash product symmetric monoidal structure exists.

\begin{proof}[Proof of \cref{prop: smash-monoidal-structure}, existence]
	By \cite[Proposition 2.2.1.9]{HA} and the previous lemma, $\sC$ inherits a nonunital symmetric monoidal structure from $\Ar(\sC)$ for which the monoidal product of $X, Y \in \sC$ is given by
	\[
	\mathrm{cofib}((\ast \to X) \otimesday (\ast \to Y)),
	\]
	which is $X \wedge Y$ by \cref{ex: formula-day-convolution-arrow}.
	Since restriction along a strong monoidal functor gives a lax monoidal functor with respect to Day convolution, the functor $\ev_1 \colon \Ar(\sC) \to \sC$ is lax monoidal from $\otimesday$ to $\otimes$.
	Composing this with the inclusion $\sC \hookrightarrow \Ar(\sC)$, which is lax monoidal from $\wedge$ to $\otimesday$ by construction of $\wedge$, we obtain the required lax monoidal functor $(\sC, \wedge) \to (\sC, \otimes)$.
\end{proof}

\begin{remark}
    If the tensor product on $\sC$ preserves all small weakly contractible colimits in both variables instead of just the finite ones, then the smash product preserves all small colimits in both variables.
\end{remark}

In order to prove the universal property of the smash product, we first show the construction is suitably natural.
We can always arrange for all symmetric monoidal $\infty$-categories we deal with to be small by enlarging the universe.
Let $\mathrm{SymMon}^{\mathrm{lax, w.c.}}_*$ denote the category with objects small nonunital symmetric monoidal $\infty$-categories $(\sC, \otimes)$ such that $\sC$ is pointed, admits all finite colimits, and the tensor product commutes with weakly contractible colimits in both variables and satisfies $\ast \otimes \ast \simeq \ast$ and with morphisms lax monoidal functors such that the underlying functor preserves the zero object.

\begin{lemma}
    The construction $(\sC, \otimes) \mapsto (\sC, \wedge)$ extends to an endofunctor of the $\infty$-category $\mathrm{SymMon}^{\mathrm{lax, w.c.}}_*$, such that the lax monoidal functors $(\sC, \wedge) \to (\sC, \otimes)$ become the components of a natural transformation from this functor to the identity functor.
\end{lemma}
\begin{proof}
    Since Day convolution is functorial, the assignment $(\sC, \otimes) \mapsto (\Ar(\sC), \otimesday)$ extends to a functor $\Ar \colon \mathrm{SymMon}^{\mathrm{lax, w.c.}}_* \to \mathrm{SymMon}^{\mathrm{lax, w.c.}}_*$ and the symmetric monoidal functor $\ev_1 \colon (\Ar(\sC), \otimesday) \to (\sC, \otimes)$ extends to a natural transformation from $\Ar$ to the identity functor.
    We write $\sC^{\wedge}$ for the $\infty$-operad corresponding to the nonunital symmetric monoidal category $(\sC, \wedge)$.
    In the same way we write $\Ar(\sC)^{\otimesday}$.
    By construction, $\sC^{\wedge}$ comes equipped with a fully faithful map $\sC^{\wedge} \hookrightarrow \Ar(\sC)^{\otimesday}$ of $\infty$-operads.
    Now if $\sD \in \mathrm{SymMon}^{\mathrm{lax, w.c.}}_*$ and $f \colon \sC \to \sD$ is a lax symmetric monoidal functor that preserves the zero object, then the composite
    \[
    \sC^{\wedge} \hookrightarrow \Ar(\sC)^{\otimesday} \xrightarrow{f_*} \Ar(\sD)^{\otimesday}
    \]
    factors through $\sD^\wedge \hookrightarrow \Ar(\sD)^{\otimesday}$.
    It therefore follows from \cite[Proposition A.1]{RamziYoneda} that the lax symmetric monoidal functors $(\sC, \wedge) \to (\Ar(\sC), \otimesday)$ assemble into a natural transformation; to apply the proposition from loc.\ cit.\ we use that fully faithful maps of $\infty$-operads are monomorphisms in $\mathrm{Op}_{\infty}$.
    We obtain the desired natural transformation by composing with $\ev_1 \colon \Ar \to \id$.
\end{proof}

\begin{remark} \label{rem: right-exact-strong-monoidal-smash}
    It follows from this construction that if $F \colon \sC \to \sD$ is a strong nonunital symmetric monoidal functor that preserves finite colimits, then the induced lax nonunital symmetric monoidal functor on smash products is strong as well.
\end{remark}

    We write $S \colon \mathrm{SymMon}^{\mathrm{lax, w.c.}}_* \to \mathrm{SymMon}^{\mathrm{lax, w.c.}}_*$ and $\alpha \colon S \to \id$ for the functor and natural transformation constructed in this lemma.
    We now prove the universal property of the smash product.

\begin{proof}[Proof of \cref{prop: smash-product-universal-property}]
    Suppose that $(\sC, \otimes)$ is a pointed nonunital symmetric monoidal $\infty$-category for which the tensor product commutes with all finite colimits in both variables separately.
    For every $X \in \sC$ we then have $X \otimes \ast \simeq \ast \simeq \ast \otimes X$, so that $X \wedge Y \simeq X \otimes Y$ and the lax symmetric monoidal functor $\alpha(\sC) \colon (\sC, \wedge) \to (\sC, \otimes)$ is an equivalence.
    Conversely, for every $(\sC, \otimes) \in \mathrm{SymMon}^{\mathrm{lax, w.c.}}_*$, the pointed nonunital symmetric monoidal category $(\sC, \wedge)$ has the property that its tensor product commutes with all finite colimits.
    The essential image of $S$ therefore consists precisely of those pointed nonunital symmetric monoidal $\infty$-categories satisfying this property.
    This means that to prove the universal property of the smash product, it will suffice to show that $\alpha$ exhibits the essential image of $S$ as a colocalization of $\mathrm{SymMon}^{\mathrm{lax, w.c.}}_*$.
    
    By \cite[Proposition 5.2.7.4]{HTT}, this follows if for every $(\sC, \otimes) \in \mathrm{SymMon}^{\mathrm{lax, w.c.}}_*$ the maps
    \[
    \alpha(S(\sC)), S(\alpha(\sC)) \colon SS(\sC, \otimes) \to S(\sC, \otimes)
    \]
    are both equivalences.
    The first of these we just checked.
    For the second, observe that we have a commutative diagram of lax symmetric monoidal functors
    \[
    \begin{tikzcd}
        \Ar(S(\sC, \otimes)) \ar[r, "\alpha(\sC)_*"] & \Ar(\sC, \otimes) \ar[d, "\mathrm{cofib}"] \\
        SS(\sC, \otimes) \ar[r, "S(\alpha(\sC))"] \ar[u] & S(\sC, \otimes).
    \end{tikzcd}
    \]
    It follows that for $X, Y \in \sC$, the lax comparison map of $S(\alpha(\sC))$ is given by the map induced on cofibers in the square
    \[
    \begin{tikzcd}
        X \otimes \ast \sqcup \ast \otimes Y \ar[r] \ar[d] & X \otimes Y \ar[d] \\
        \ast \ar[r] & X \wedge Y,
    \end{tikzcd}
    \]
    which is an isomorphism.
    Therefore, $S(\alpha(\sC))$ is a strong monoidal functor whose underlying functor is the identity, showing that it is an equivalence.
\end{proof}

The uniqueness of the smash product symmetric monoidal structure follows easily from this universal property:

\begin{proof}[Proof of \cref{prop: smash-monoidal-structure}, uniqueness]
    Suppose that $(\sC, \wedge')$ is a symmetric monoidal structure on $\sC$ satisfying (1) and (2) from \cref{prop: smash-monoidal-structure}.
    Then the tensor product $- \wedge' - \colon \sC \times \sC \to \sC$ is the smash product by (1), so that it preserves finite colimits in both variables by \cref{prop: smash-preserves-colimits}.
    It follows from (2) that the canonical map $X \otimes Y \to X \wedge' Y$ extends to a lax monoidal functor $(\sC, \wedge') \to (\sC, \otimes)$.
    The universal property of the smash product now yields a unique lax symmetric monoidal functor $(\sC, \wedge') \to (\sC, \wedge)$ such that the diagram
    \[
    \begin{tikzcd}[sep = small]
        & X \otimes Y \ar[dr] \ar[dl] & \\
        X \wedge Y \ar[rr] & & X \wedge' Y
    \end{tikzcd}
    \]
    commutes.
    But then the lax comparison map $X \wedge Y \to X \wedge' Y$ is an equivalence, so that $(\sC, \wedge') \to (\sC, \wedge)$ is a strong monoidal functor with underlying functor the identity, which implies that it is an equivalence of symmetric monoidal $\infty$-categories.
\end{proof}

It sometimes is useful to consider the smash product even in situations where the tensor product does not commute with weakly contractible colimits.
In this generality, the smash product defines what is known as a nonunital oplax symmetric monoidal $\infty$-category.
Recall that this is a nonunital $\infty$-operad $\sD^\otimes$, such that the structure morphism
\[
\sD^\otimes \to \Surj_*
\]
is a locally cocartesian fibration; here $\Surj_*$ is the category of finite pointed sets and surjections.
We will usually abuse notation by calling $\sD$ itself a nonunital oplax symmetric monoidal $\infty$-category.
For any non-empty finite set $I$ and $I$-indexed set of objects $\{X_i\}_{i \in I}$ in $\sD$, we obtain an $I$-fold tensor product
\[
\otimes^I\{X_i\}_{i \in I} \in \sD
\]
by choosing a locally cartesian lift of the active morphism $I_+ \to \langle 1 \rangle$.
These $I$-fold tensor products are related by associator morphisms: for every partition $E$ of $I$,
we obtain a natural map
\[
\wedge^I\{X_i\}_{i \in I} \to \wedge^E\{\wedge^J\{X_j\}_{j \in J}\}_{J \in E}.
\]
The nonunital oplax symmetric monoidal $\infty$-category $\sD$ is an honest nonunital symmetric monoidal $\infty$-category if and only if all these maps are equivalences.

If $\sC$ and $\sD$ are nonunital oplax symmetric monoidal $\infty$-categories, we define a lax symmetric monoidal functor $\sC \to \sD$ to be a map of $\infty$-operads $\sC^\otimes \to \sD^\otimes$.
We define a strong monoidal functor to be a map of $\infty$-operads $\sC^\otimes \to \sD^\otimes$ that preserves locally cocartesian morphisms.

To show the smash product defines a nonunital oplax symmetric monoidal $\infty$-category, we use the following proposition.

\begin{proposition} \label{prop: localization-oplax-symmetric-monoidal-category}
    Let $\sC$ be a nonunital oplax symmetric monoidal $\infty$-category and let
    \[
    \begin{tikzcd}
    L \colon \sC \ar[r, shift left] & \sD \ar[l, shift left, hook'] \colon \iota
    \end{tikzcd}
    \]
    be a reflective localization of $\sC$.
    Then $\sD$ can be made into a nonunital oplax symmetric monoidal $\infty$-category such that for any non-empty set $I$ and $I$-indexed collection of objects $\{X_i\}_{i \in I}$ of $\sD$, we have
    \[
    \otimes^I_{\sD}\{X_i\}_{i \in I} \simeq L(\otimes^I_{\sC}\{X_i\}_{i \in I}).
    \]
    Moreover, $\iota$ can be made into a lax symmetric monoidal functor.
\end{proposition}
\begin{proof}
    This follows from the proof of \cite[Proposition A.3.25]{blansblom2025chainrulegoodwilliecalculus}.
\end{proof}

Now suppose that $\sC$ is a nonunital oplax symmetric monoidal $\infty$-category which admits finite colimits, is pointed, and satisfies $\otimes^I \{*\}_{i \in I} \simeq \ast$ for any non-empty finite set $I$, i.e.\  $I$-fold tensor products of the zero object are the zero object.
For $I$ a non-empty finite set and $\{X_i\}_{i \in I}$ a collection of objects in $\sC$, we define the $I$-fold smash product $\wedge^I \{X_i \}_{i \in I}$ to be the total cofiber of the $I$-cube in $\sC$ defined by
\[
S \mapsto \otimes^I\{Y_i\}_{i \in I} \qquad Y_i = 
\begin{cases}
    X_i & \text{if $i \in S$} \\
    * & \text{else}
\end{cases}
\]
for $S$ a subset of $I$.

By repeating our earlier constructions, the $I$-fold smash products assemble to give a nonunital oplax symmetric monoidal structure on $\sC$.
To be precise, we first observe that Day convolution gives $\Ar(\sC)$ the structure of a nonunital oplax symmetric monoidal $\infty$-category, and then apply \cref{prop: localization-oplax-symmetric-monoidal-category} to the localization $\mathrm{cofib} \colon \Ar(\sC) \to \sC$.
Slight modifications of the arguments in this appendix then give the following generalizations of \cref{prop: smash-monoidal-structure} and \cref{prop: smash-product-universal-property}.

\begin{proposition} \label{prop: smash-oplax-monoidal-structure}
    Suppose that $(\sC, \otimes^I)$ is a nonunital oplax symmetric monoidal $\infty$-category which admits finite colimits, is pointed, and satisfies $\otimes^I \{*\}_{i \in I} \simeq \ast$ for any non-empty finite set $I$.
    Then there exists an essentially unique nonunital oplax symmetric monoidal structure on $\sC$ for which
    \begin{enumerate}[\upshape{(}\arabic*\upshape{)}]
        \item for any non-empty finite set $I$, the $I$-fold monoidal product is given by $\wedge^I$, and
        \item the canonical morphisms $\otimes^I\{X_i\}_{i \in I} \to \wedge^I \{X_i\}_{i \in I}$ extend to a lax symmetric monoidal functor $(\sC, \wedge^I) \to (\sC, \otimes^I)$ which lifts the identity functor of $\sC$.
    \end{enumerate}
\end{proposition}

\begin{proposition} \label{prop: oplax-smash-product-universal-property}
    Let $(\sC, \otimes^I_\sC)$ be a nonunital oplax symmetric monoidal $\infty$-category as in the previous proposition.
    Suppose $(\sD, \otimes^I_\sD)$ is a nonunital oplax symmetric monoidal $\infty$-category such that for every non-empty finite set $I$ and every collection of objects $\{X_i\}_{i \in I}$ in $\sD$, we have $\otimes^I_\sD\{X_i\}_{i \in I} = \ast$ whenever $X_i \simeq \ast$ for some $i \in I$.
    Then postcomposition with the lax symmetric monoidal functor $(\sC, \wedge^I) \to (\sC, \otimes_\sC^I)$ induces an equivalence
    \[
    \Fun_*^{\mathrm{lax}}((\sD, \otimes_\sD^I), (\sC, \wedge^I)) \xrightarrow{\sim} \Fun_*^{\mathrm{lax}}((\sD, \otimes_\sD^I), (\sC, \otimes_\sC^I)),
    \]
    where $\Fun_*^{\mathrm{lax}}$ denotes the $\infty$-category of lax symmetric monoidal functors whose underlying functor preserves the zero object.
\end{proposition}

We end with the following criterion for showing that the smash product defined with respect to the cartesian product defines an honest nonunital symmetric monoidal category.

\begin{proposition} \label{prop: smash-product-T-algebras-associative}
    Let $(\sC, \otimes)$ be a stable presentably symmetric monoidal $\infty$-category and let $T$ be a sifted colimit preserving monad on $\sC$.
    Suppose that the free $T$-algebra functor
    \[
    \free_T \colon \sC \to \Alg_T(\sC)
    \]
    admits a nonunital symmetric monoidal structure from $\otimes$ to $\wedge$, where the smash product is taken with respect to the cartesian product on $\Alg_T(\sC)$.
    Then the smash product defines an honest nonunital symmetric monoidal structure on $\Alg_T(\sC)$.
\end{proposition}
\begin{proof}
    For each non-empty finite set $I$, the functor
    \[
    (X_i)_{i \in I} \mapsto \wedge^I\{X_i\}_{i \in I}
    \]
    on $\Alg_T(\sC)$ preserves sifted colimits; this holds since sifted colimits and products are both calculated in the stable category $\sC$, where they commute.
    To prove that the smash product is associative, it suffices to show for every non-empty finite set $I$ and every partition $E$ of $I$ the natural transformation
    \[
    \wedge^I\{X_i\}_{i \in I} \to \wedge^E\{\wedge^J\{X_j\}_{j \in J}\}_{J \in E}
    \]
    is an equivalence.
    Using that every $T$-algebra is a sifted colimit of free $T$-algebras, we can assume that each $X_i$ is a free $T$-algebra.
    But in this case the map is an equivalence since $\free_T$ is symmetric monoidal and the tensor product on $\sC$ is associative.
\end{proof}

\phantomsection
\printbibliography 

\end{document}